\pdfoutput=1
\documentclass[reqno]{amsart}
\usepackage{amsmath,amsthm,amssymb,mathtools}
\IfFileExists{newtxtext.sty}{\usepackage{newtxtext,newtxmath}}{\usepackage{txfonts}}
\usepackage{microtype}
\usepackage{graphicx,booktabs,array,enumitem,xcolor,placeins}
\graphicspath{{figures/}}
\usepackage{caption}
\usepackage{subcaption}
\usepackage[numbers,sort&compress]{natbib}
\usepackage[colorlinks,linkcolor=blue,citecolor=blue,urlcolor=blue]{hyperref}
\usepackage[nameinlink,noabbrev,capitalise]{cleveref}
\numberwithin{equation}{section}
\theoremstyle{plain}
\newtheorem{theorem}{Theorem}[section]
\newtheorem{proposition}[theorem]{Proposition}
\newtheorem{lemma}[theorem]{Lemma}
\newtheorem{corollary}[theorem]{Corollary}
\theoremstyle{definition}
\newtheorem{assumption}[theorem]{Assumption}
\theoremstyle{remark}
\crefname{assumption}{Assumption}{Assumptions}
\Crefname{assumption}{Assumption}{Assumptions}
\newcommand{\E}{\mathbb E}
\newcommand{\R}{\mathbb R}
\newcommand{\A}{\mathcal A}
\newcommand{\C}{\mathcal C}
\newcommand{\calU}{\mathcal U}

\DeclareMathOperator{\sgn}{sgn}
\begin{document}

\title[Consumption adjustment with mean-reverting risk premia]{Consumption adjustment with mean-reverting risk premia: A singular stochastic control approach}

\author{Gugyum Ha}
\address{Department of Mathematics, Sogang University, 35 Baekbeom-ro, Mapo-gu, Seoul 04107, Republic of Korea}
\email{ggha@sogang.ac.kr}

\author{Junkee Jeon}
\address{Department of Applied Mathematics, Kyung Hee University, 1732 Deogyeong-daero, Giheung-gu, Yongin-si, Gyeonggi-do 17104, Republic of Korea}
\email{junkeejeon@khu.ac.kr}

\author{Jehan Oh}
\address{Department of Mathematics, Kyungpook National University, 80 Daehak-ro, Buk-gu, Daegu 41566, Republic of Korea}
\email{jehan.oh@knu.ac.kr}


\begin{abstract}
We study a two-sided singular stochastic control problem in which changes in consumption carry utility costs and the risk premium follows the mean-reverting Kim--Omberg model. Optimal consumption is reflected between two moving barriers given by thresholds that depend on the current risk premium. These thresholds are the stopping boundaries of a Dynkin game for a two-dimensional diffusion driven by a single Brownian motion. A sharp occupation estimate shows that they are Lipschitz in logarithmic scale, uniformly in the cost of reducing consumption, and that smooth fit holds along the diffusion direction at every contact point. Outside a bounded interval of risk premia and at most two further points, the thresholds are infinitely differentiable. As the reduction cost grows, optimal consumption converges to the ratcheting policy.
\end{abstract}

\maketitle

\section{Introduction}
Households rebalance their portfolios far more often than they revise their standard of living. Choi, Jeon, and Koo \cite{CJK} model this reluctance by utility costs for upward and downward changes in consumption. With constant investment opportunities their optimal policy is described by two constant wealth--consumption ratios. Consumption is raised when the ratio exceeds the upper level, lowered when it falls below the lower level, and left unchanged in between. In this paper we keep these preferences and let the expected excess return of the risky asset follow the mean reverting diffusion of Kim and Omberg \cite{KO}. As in Wachter \cite{Wachter}, the risky asset spans the single Brownian shock. This is the market studied by Ferrari and Sch\"utz \cite{FS26}, who add labor income and a dynamic no borrowing constraint. With utility costs of adjusting consumption in place of their constraint, the two ratios of \cite{CJK} depend on the current risk premium, and the two algebraic equations that determine them are replaced by the free boundaries of a degenerate problem in two state variables.

We describe the problem by the current risk premium and the ratio of the shadow price of wealth to the marginal utility of current consumption. In these variables the marginal value of consumption solves a double obstacle equation whose obstacles are the two utility costs. Consumption is increased when the normalized shadow price falls below a lower threshold and reduced when it rises above an upper threshold. The consumption problem is a two sided singular stochastic control problem, and as for other problems of this kind \cite{BK98,Boetius,CK96,DAF14} the marginal value is the value of a Dynkin game whose stopping regions are the two contact sets (\cref{sec:potentials}). A single Brownian motion drives both state variables, so the diffusion matrix has rank one at every point. We construct the solution by monotone penalization. The penalized solutions grow linearly in the shadow price, and we compare them through their mild formulation.

The analysis of the free boundaries rests on a sharp bound for the sensitivity of the marginal value to the risk premium relative to its sensitivity to the shadow price. We derive it from a linearized penalty equation and a supersolution in one variable. It shows that both thresholds are Lipschitz in logarithmic scale, uniformly in the cost of reducing consumption. Coordinates that straighten the diffusion direction then give locally bounded second derivatives along that direction and continuous smooth fit along it at every contact point. On an explicit set of risk premia the same coordinates turn the problem into a parabolic obstacle problem in one space variable. There the derivative with respect to the shadow price solves an equation with a positive source, which gives nondegeneracy, and the parabolic boundary Harnack inequalities of Kukuljan \cite{Kukuljan} make the thresholds infinitely differentiable. All constants are uniform in the cost of reducing consumption, so the estimates carry over to the limit of infinite cost, a single obstacle problem whose solution is given by a perpetual American put on the shadow price.

The explicit set leaves out a bounded central interval $[-B_*,B_*]$ of risk premia and at most two further points. The interval always contains zero and can carry most of the stationary distribution of the premium. In the numerical example of \cref{sec:numerics} it contains the long run mean and has stationary probability $0.6544$. Whether the thresholds are continuously differentiable over it is open. The verification of the optimal strategy uses only the Lipschitz bound and the smooth fit along the diffusion, so the optimization problem is solved on the whole state space, including the premia in this interval. All results are stated for relative risk aversion greater than one, without labor income, in the complete market of \cref{sec:model}.

The degeneracy differs from that of the Heston model, whose operator is strictly elliptic at positive variance when the asset and variance shocks are imperfectly correlated. We follow the order of argument used for Heston American puts by Jeon, Kim, and Oh \cite{JKO} and in the weighted obstacle estimates of Daskalopoulos and Feehan \cite{DF}, with estimates for the solution first, then the geometry of the contact set, and finally the regularity of the boundary. Their ellipticity assumptions fail here, so the estimates are proved directly and the regions in which the higher order argument applies are identified explicitly.

The marginal equation determines the thresholds but not the level of the dual value. The two contact reactions have resolvent potentials which solve a switching system, and integrating these potentials over consumption levels gives the dual value without an undetermined function of the risk premium. A scalar minimization over the multiplier exhausts the budget, and martingale representation gives the financing portfolio. These steps adapt the layer decomposition and the duality arguments of reversible investment and singular control \cite{AE96,BK98,FP14,DAF14,KS98} to the degenerate setting, and they yield optimal consumption, wealth and investment on the whole admissible state space from the Lipschitz bound and the smooth fit alone.

When consumption may never be reduced, the marginal equation has a single obstacle. Optimal consumption is then the running maximum of a moving barrier, and the investor is solvent exactly when wealth covers the cost of keeping the current consumption level forever. As the cost of reducing consumption grows, the values decrease to the ratcheting value, the increase thresholds converge locally uniformly and the reduction thresholds escape to infinity. Optimal consumption converges uniformly on compact time intervals, first for a fixed multiplier and then, because the optimal multipliers converge, for a fixed initial wealth above the solvency floor. With constant investment opportunities an explicit calculation shows that the reduction threshold in wealth units converges to a level strictly below the solvency floor, so the limit of the thresholds differs from the threshold of the limit.

Predictable returns also affect the portfolio. It contains a hedge against changes in investment opportunities which remains when consumption is adjusted, so the result of \cite{CJK} that the risky share at an adjustment threshold equals the myopic share of a CRRA investor does not carry over. The normalized intervention problem does not depend on risk aversion, while its translation into consumption and wealth does.

Our work builds on the consumption and investment problems of Merton \cite{M69,M71} and on the martingale approach for complete markets \cite{CH89,KLS87,KS98}. Predictable investment opportunities are studied in \cite{KO,Wachter,Barberis}, and consumption commitments and illiquid consumption goods in \cite{GL90,CS07,CS16}, where adjustment costs are paid in cash, whereas ours are measured in utility. Consumption ratcheting goes back to Dybvig \cite{Dybvig} and is studied further in \cite{Riedel,JKS18}, and related drawdown and shortfall aversion formulations appear in \cite{ABY19,BHYY26,DLPY22,GHR20,LYZ24}.

On the mathematical side, the connections among reversible investment, singular control, switching and stopping games are developed in \cite{AE96,BK98,Boetius,CK96,FP14}. Closest to our setting, De Angelis and Ferrari \cite{DAF14} characterize a partially reversible investment problem on a finite horizon through a Dynkin game with two moving boundaries. Smooth fit and continuity of stopping boundaries are studied in \cite{Jacka,DAP,DAStabile}, and general stopping theory and weak generators in \cite{PS06,Peskir25}. We use comparison and stability for viscosity solutions \cite{CIL,FS06}, interior hypoellipticity \cite{Hormander}, and linear parabolic boundary estimates \cite{Lieberman,Kukuljan}. Classical and recent regularity results for obstacle problems are developed in \cite{KN,CPS,BDM,AK}. They require specific solution classes and geometric hypotheses, and the difference between intrinsic and Euclidean regularity is important for kinetic operators \cite{Bowman}. Analytic studies of exercise boundaries in option pricing include \cite{ChenChadam}. The reflection construction is related to \cite{BKR}, and the numerical approximations to \cite{BS91,KD01}.

\Cref{sec:model} introduces the model, the marginal obstacle equation and the occupation estimate. \Cref{sec:regularity} studies the intervention thresholds, from the Lipschitz bound to smoothness on explicit regions. \Cref{sec:strategy} reconstructs the dual value, verifies optimal consumption, wealth and investment, and treats ratcheting and the limit of infinite reduction cost. \Cref{sec:numerics} presents numerical illustrations and \cref{sec:conclusion} concludes. The appendices contain integrability details, the discrete reconstruction used in the computations, and remarks on the regularity methods.

\section{The model and the marginal obstacle problem}\label{sec:model}
\subsection{Market, preferences, and the budget constraint}
Let $W$ be a Brownian motion and let the filtration be the augmented filtration it generates. The money market account, the risky asset and the expected excess return satisfy
\begin{align}
 \frac{dB_t}{B_t}&=rdt, &
 \frac{dS_t}{S_t}&=(r+b_t)dt+\sigma dW_t,\label{eq:market}\\
 db_t&=\kappa(\bar b-b_t)dt-\eta dW_t.
 \label{eq:factor}
\end{align}
\begin{assumption}\label{ass:main}
We have $r,\delta,\sigma,\kappa,\eta>0$, $\kappa\sigma>\eta$, and $\bar b\in\R$. Preferences satisfy $\gamma>1$, $0\leq\alpha<1/\delta$, $0\leq\lambda<\infty$, and $\alpha+\lambda>0$.
\end{assumption}
Under the pricing measure the factor reverts to its mean at speed $\kappa-\eta/\sigma$, which is positive under \cref{ass:main}. Apart from the uniqueness argument in \cref{prop:ratchetPDE}, which uses $\kappa\sigma>\eta$, the PDE arguments below only use $2\kappa\sigma>\eta$. The current value of $b_t$ may have either sign. The investor may take long or short positions and lever the portfolio, and the only solvency requirement is that total wealth stays nonnegative.

For a dollar amount $\pi_t$ invested in the risky asset and a consumption rate $c_t$, wealth follows
\begin{equation}
 dX_t=(rX_t+b_t\pi_t-c_t)dt+\sigma\pi_tdW_t,
 \qquad X_0=x>0.\label{eq:wealth}
\end{equation}
Initial consumption is $c_{0-}=c>0$. Set
\begin{equation}
 u(c)=\frac{c^{1-\gamma}}{1-\gamma},\qquad q_t=u(c_t),
 \qquad q_t=q_{0-}+q_t^+-q_t^-,\label{eq:utility}
\end{equation}
where $q^+$ and $q^-$ are the Jordan variations, including a possible adjustment at time zero. The criterion is
\begin{equation}
 \calU_{\alpha,\lambda}(c_\cdot)
 =\E\left[\int_0^\infty e^{-\delta t}u(c_t)dt
 -\int_{[0,\infty)}e^{-\delta t}
       (\alpha dq_t^++\lambda dq_t^-)\right].\label{eq:objective}
\end{equation}
At an upward jump from $c_-$ to $c_+$ the cost is $\alpha[u(c_+)-u(c_-)]$, which is the marginal cost $\alpha u'$ integrated over the jump, and a downward jump is treated in the same way.

Admissible consumption is adapted, positive, c\`adl\`ag, and locally of bounded variation. We require a finite present value of consumption and finite expected discounted negative utility and adjustment costs. When one of the cost coefficients is zero, the last condition is imposed on $(\alpha+\lambda)dq^-$. The terminal condition
\begin{equation}
 \lim_{T\to\infty}\E[e^{-\delta T}|u(c_T)|]=0\label{eq:utilityTV}
\end{equation}
is part of admissibility, and the integrability estimates in \cref{sec:pricingappendix} show that the constructed optimizer satisfies it. Portfolios are locally square integrable and wealth is nonnegative. We write $V_\lambda(x,c,b)$ for the supremum of \eqref{eq:objective} over these controls.

The state-price density and the uncontrolled shadow price process are
\begin{align}
 H_t&=e^{-rt}M_t, &
 M_t&=\mathcal E\left(-\int_0^t\frac{b_s}{\sigma}dW_s\right),\label{eq:spd}\\
 Y_t&=y e^{\delta t}H_t, &
 dY_t&=(\delta-r)Y_tdt-\frac{b_t}{\sigma}Y_tdW_t.
 \label{eq:shadow}
\end{align}
The exponential $M$ is a martingale on every finite horizon (\cref{sec:pricingappendix}), so that
\begin{equation}
 \E_b[H_t]=e^{-rt},\qquad
 \E_{z,b}\int_T^\infty e^{-\delta t}Z_tdt
       =\frac{z e^{-rT}}r,\label{eq:linear-moment}
\end{equation}
where $Z$ follows the same dynamics as $Y$ and starts from $z$.

\begin{proposition}\label{prop:budget}
A positive adapted consumption stream can be financed from $x$ with nonnegative wealth if and only if
\begin{equation}
 \E\int_0^\infty H_t c_tdt\leq x.\label{eq:budget}
\end{equation}
When equality holds, its minimal financing wealth is
\begin{equation}
 X_t=\frac1{H_t}\E_t\int_t^\infty H_s c_sds.\label{eq:financingwealth}
\end{equation}
\end{proposition}
\begin{proof}
It\^o's formula gives
$d(H_tX_t)=-H_tc_tdt+H_t(\sigma\pi_t-b_tX_t/\sigma)dW_t$, and localization and nonnegativity imply \eqref{eq:budget}. Conversely, represent the uniformly integrable martingale
$\mathcal N_t=\E_t\int_0^\infty H_sc_sds$ as $d\mathcal N_t=\psi_tdW_t$.
Then \eqref{eq:financingwealth}, together with
$\pi_t=\sigma^{-1}(\psi_t/H_t+b_tX_t/\sigma)$, satisfies \eqref{eq:wealth} and is nonnegative. A strict budget surplus can be held in the money market account. Brownian martingale representation and localization give the required local integrability \cite{KS98,Protter}.
\end{proof}

\subsection{Utility units and normalization of the costs}
Write $C(q)=u^{-1}(q)=((1-\gamma)q)^{1/(1-\gamma)}$ for $q<0$. Then
\begin{equation}
 C'(q)=C(q)^\gamma,\qquad C''(q)=\gamma C(q)^{2\gamma-1}>0.
 \label{eq:inverseutility}
\end{equation}
Consequently $q-yC(q)$ is strictly concave in $q$. This is the concavity used in the dual problem, since the original objective need not be concave in consumption rates. The normalization of the costs below and the failure of concavity in consumption rates are also discussed in \cite{CJK25}, which works in the same utility units.

Integration by parts, using \eqref{eq:utilityTV}, gives
\begin{equation}
 \calU_{\alpha,\lambda}
 =\alpha u(c)+(1-\delta\alpha)\E\int_0^\infty e^{-\delta t}q_tdt
 -(\alpha+\lambda)\E\int_{[0,\infty)}e^{-\delta t}dq_t^-.
 \label{eq:normalization}
\end{equation}
With $A_0=1-\delta\alpha>0$ and $\hat\lambda=(\alpha+\lambda)/A_0$ we obtain
\begin{equation}
 \calU_{\alpha,\lambda}=\alpha u(c)+A_0\calU_{0,\hat\lambda},\qquad
 \mathrm{ILA}=1+\delta\hat\lambda=\frac{1+\delta\lambda}{1-\delta\alpha}.
 \label{eq:ila}
\end{equation}
Thus two pairs of costs with the same coefficient of intertemporal loss aversion $\mathrm{ILA}$ of \cite{CJK,CJK25} generate the same primal policy, as in the model with constant investment opportunities. \Cref{prop:reduction} gives the corresponding statement for the marginal equation.

\subsection{The marginal obstacle equation}\label{sec:pde}
For a fixed multiplier $y>0$, define the dual consumption problem
\begin{equation}
 J_\lambda(y,c,b)=\sup_{c_\cdot}\E\left[\int_0^\infty e^{-\delta t}
 (u(c_t)-Y_tc_t)dt-\int_{[0,\infty)}e^{-\delta t}
 (\alpha dq_t^++\lambda dq_t^-)\right].\label{eq:dual}
\end{equation}
No wealth budget is imposed here. A constant consumption stream gives a finite lower bound, and for $\gamma>1$ the objective is nonpositive. We construct the solution from a marginal PDE and establish its relation to \eqref{eq:dual} in \cref{sec:strategy}.

For a consumption level $a>0$, the marginal flow payoff divided by $u'(a)$ is $1-ya^\gamma$. This motivates the state $z=ya^\gamma$ and the uncontrolled diffusion
\begin{equation}
 dZ_t=(\delta-r)Z_tdt-\frac{b_t}{\sigma}Z_tdW_t,
 \qquad db_t=\kappa(\bar b-b_t)dt-\eta dW_t,\label{eq:uncontrolledZ}
\end{equation}
with generator
\begin{equation}
 \A=\frac{b^2z^2}{2\sigma^2}\partial_{zz}
 +\frac{\eta bz}{\sigma}\partial_{zb}+\frac{\eta^2}{2}\partial_{bb}
 +(\delta-r)z\partial_z+\kappa(\bar b-b)\partial_b.\label{eq:generator}
\end{equation}
The mixed term is positive because both coordinates load negatively on $W$. Write
\begin{equation}
 A_0=1-\delta\alpha>0,\quad B_\lambda=1+\delta\lambda,
 \quad K=\frac{A_0}\delta,\quad g(z)=\frac1\delta-\frac zr.\label{eq:constants}
\end{equation}
Our starting point is the double obstacle equation
\begin{equation}
 \min\{\max[\delta v-\A v-1+z,\ v-\alpha],\ v+\lambda\}=0
 \quad\text{on }(0,\infty)\times\R.\label{eq:doubleVI}
\end{equation}
Solutions are understood in the viscosity sense, and the constraint $-\lambda\leq v\leq\alpha$ is part of the equation. The left side of \eqref{eq:doubleVI} is nondecreasing in each of the three expressions $\delta v-\A v-1+z$, $v-\alpha$ and $v+\lambda$, which we use repeatedly to produce sub and supersolutions.

\subsection{Construction by monotone penalization}
\begin{proposition}\label{prop:PDEexist}
Equation \eqref{eq:doubleVI} has a unique bounded continuous solution $v_\lambda$. It is the locally uniform limit of the penalized problems below and satisfies
\begin{equation}
 \alpha-\frac zr\leq v_\lambda(z,b)\leq\alpha,\qquad
 -\lambda\leq v_\lambda(z,b).\label{eq:lowzbound}
\end{equation}
\end{proposition}
\begin{proof}
We first solve penalized problems and compare them through their mild formulation. For $T>0$ and $\epsilon>0$ consider
\begin{equation}
\begin{gathered}
 V^{\epsilon,T}_t+\A V^{\epsilon,T}-\delta V^{\epsilon,T}
 +1-z-\beta_\epsilon(V^{\epsilon,T})=0,
 \qquad V^{\epsilon,T}(T,\cdot)=0,\\
 \beta_\epsilon(q)
 =\frac{(q-\alpha)^+-(-\lambda-q)^+}{\epsilon}.
\end{gathered}
 \label{eq:penaltyPDE}
\end{equation}
With $\mathcal S_t$ the semigroup of \eqref{eq:uncontrolledZ}, the mild form of \eqref{eq:penaltyPDE} with a terminal value $\varphi$ in place of zero is
\begin{equation}
 V(t,\cdot)
 =e^{-\delta(T-t)}\mathcal S_{T-t}\varphi+\int_t^T e^{-\delta(s-t)}\mathcal S_{s-t}
 \bigl[1-z-\beta_\epsilon(V(s,\cdot))\bigr]ds.
 \label{eq:mild}
\end{equation}
Let $\mathcal C_1$ be the space of continuous functions on $(0,\infty)\times\R$ with finite norm $\|f\|_1=\sup|f|/(1+z)$. The semigroup $\mathcal S_t$ is positive, and $\mathcal S_t1=1$ and $\mathcal S_tz=ze^{(\delta-r)t}$ by \eqref{eq:linear-moment}, so $\|\mathcal S_tf\|_1\leq e^{|\delta-r|t}\|f\|_1$.

Comparison for the penalized problems is carried out on the mild formulation, which applies to functions of linear growth in $z$. Put $K_\epsilon=1/\epsilon$. Since $\beta_\epsilon$ is nondecreasing with Lipschitz constant $1/\epsilon$, the map $q\mapsto K_\epsilon q-\beta_\epsilon(q)$ is nondecreasing and $K_\epsilon$-Lipschitz. For $\varphi\in\mathcal C_1$ and $V:[0,T]\to\mathcal C_1$ define
\begin{equation}
 \mathcal M_\varphi(V)(t)=e^{-(\delta+K_\epsilon)(T-t)}\mathcal S_{T-t}\varphi
 +\int_t^Te^{-(\delta+K_\epsilon)(s-t)}\mathcal S_{s-t}
 \bigl[1-z+K_\epsilon V(s)-\beta_\epsilon(V(s))\bigr]ds.
 \label{eq:shiftedmild}
\end{equation}
For $\varphi\in\mathcal C_1$ and bounded measurable $f:[0,T]\to\mathcal C_1$, let $\mathcal W$ be defined by the first line of
\begin{align}
 \mathcal W(t)&=e^{-\delta(T-t)}\mathcal S_{T-t}\varphi+\int_t^Te^{-\delta(s-t)}\mathcal S_{s-t}f(s)ds\nonumber\\
 &=e^{-(\delta+K_\epsilon)(T-t)}\mathcal S_{T-t}\varphi+\int_t^Te^{-(\delta+K_\epsilon)(s-t)}\mathcal S_{s-t}\bigl[f(s)+K_\epsilon\mathcal W(s)\bigr]ds.
 \label{eq:Kshift}
\end{align}
The second line follows from the semigroup property, Fubini's theorem and the identity
\[
 \int_t^uK_\epsilon e^{-K_\epsilon(s-t)}ds=1-e^{-K_\epsilon(u-t)}.
\]
With $f=1-z-\beta_\epsilon(V)$, this shows that a solution $V$ of \eqref{eq:mild} with terminal value $\varphi$ satisfies $V=\mathcal M_\varphi(V)$. Conversely, if $V=\mathcal M_\varphi(V)$ and $\widetilde V$ denotes the right side of \eqref{eq:mild}, then \eqref{eq:Kshift} shows that $V-\widetilde V$ solves the linear equation $V-\widetilde V=\int_t^Te^{-(\delta+K_\epsilon)(s-t)}\mathcal S_{s-t}[K_\epsilon(V-\widetilde V)(s)]ds$, whose only solution is zero by the contraction estimate below. In the norm $\sup_{t\leq T}e^{-\vartheta(T-t)}\|V(t)\|_1$ on bounded measurable maps $V:[0,T]\to\mathcal C_1$, with $\vartheta>|\delta-r|$, the map $\mathcal M_\varphi$ is a contraction with constant $K_\epsilon/(K_\epsilon+\delta+\vartheta-|\delta-r|)<1$. Hence \eqref{eq:mild} has a unique solution with values in $\mathcal C_1$, which we denote by $V^{\epsilon,T}$ when $\varphi=0$, and the Picard iterates of $\mathcal M_\varphi$ converge to it from any starting point \cite[Chapter~6, Section~1]{Pazy}. The semigroup property and Dynkin's formula show that it is a viscosity solution of \eqref{eq:penaltyPDE} \cite[Chapter~V, Section~3]{FS06}.

The map $\mathcal M_\varphi$ is order preserving, because $\mathcal S$ is positive and $q\mapsto K_\epsilon q-\beta_\epsilon(q)$ is nondecreasing, and it is nondecreasing in $\varphi$. We call $\widetilde V$ a mild supersolution with terminal value $\varphi$ if $\widetilde V\geq\mathcal M_\varphi(\widetilde V)$ with $\mathcal M_\varphi$ as in \eqref{eq:shiftedmild}, and a mild subsolution if the reverse inequality holds. If $\widetilde V$ is a mild supersolution with terminal value $\varphi$ and $V$ solves \eqref{eq:mild} with terminal value $\psi\leq\varphi$, then $\mathcal M_\psi(\widetilde V)\leq\widetilde V$, the iterates $\mathcal M_\psi^n(\widetilde V)$ decrease and converge to $V$, and therefore $V\leq\widetilde V$. Subsolutions are treated in the same way.

For two kinds of functions, mild sub and supersolutions can be recognized directly. If $\widetilde V(t,z,b)=c_0(t)+c_1(t)z$ with continuously differentiable coefficients, then $\mathcal S_t1=1$ and $\mathcal S_tz=ze^{(\delta-r)t}$ give
\[
 \widetilde V(t)=e^{-\delta(T-t)}\mathcal S_{T-t}\widetilde V(T)
 +\int_t^Te^{-\delta(s-t)}\mathcal S_{s-t}\bigl[(-\partial_s-\A+\delta)\widetilde V(s)\bigr]ds,
\]
and by \eqref{eq:Kshift} and positivity, $(-\partial_t-\A+\delta)\widetilde V\geq1-z-\beta_\epsilon(\widetilde V)$ together with $\widetilde V(T)\geq\varphi$ makes $\widetilde V$ a mild supersolution with terminal value $\varphi$. The reverse inequalities give a mild subsolution. For functions of the form $\mathcal W$ in \eqref{eq:Kshift}, the second line of \eqref{eq:Kshift} plays the same role.

Next we prove the uniform bounds
\begin{equation}
 -\lambda-\epsilon z\leq V^{\epsilon,T}\leq\alpha+\epsilon.
 \label{eq:penaltybounds}
\end{equation}
Indeed, $\underline V=-\lambda-\epsilon z$ and $\overline V=\alpha+\epsilon$ are affine in $z$ and satisfy $\beta_\epsilon(\underline V)=-z$ and $\beta_\epsilon(\overline V)=1$. Since $\A z=(\delta-r)z$,
\begin{align*}
 (-\partial_t-\A+\delta)\underline V&=-\delta\lambda-\epsilon rz\leq1=1-z-\beta_\epsilon(\underline V),\\
 (-\partial_t-\A+\delta)\overline V&=\delta(\alpha+\epsilon)\geq-z=1-z-\beta_\epsilon(\overline V),
\end{align*}
and $\underline V\leq0\leq\overline V$ at $t=T$. By the criterion for affine functions above, $\underline V$ and $\overline V$ are a mild subsolution and a mild supersolution with terminal value zero, which gives \eqref{eq:penaltybounds}. Next let
\[
 \mathcal W(t)=\int_t^Te^{-\delta(s-t)}\mathcal S_{s-t}|1-z|ds\geq0.
\]
Since $\beta_\epsilon(0)=0$ and $\beta_\epsilon$ is nondecreasing, $1-z-\beta_\epsilon(\mathcal W)\leq|1-z|$ and $1-z-\beta_\epsilon(-\mathcal W)\geq-|1-z|$, so \eqref{eq:Kshift} gives $\mathcal M_0(\mathcal W)\leq\mathcal W$ and $\mathcal M_0(-\mathcal W)\geq-\mathcal W$. Hence
\[
 |V^{\epsilon,T}(t,\cdot)|
 \leq
 \int_0^{T-t}e^{-\delta s}\mathcal S_s|1-z|ds,
\]
and the right side tends to zero as $t\nearrow T$, locally uniformly in $(z,b)$ and uniformly in $\epsilon$.

We now let $\epsilon\searrow0$ for fixed $T$. By \eqref{eq:penaltybounds} the family $\{V^{\epsilon,T}\}_{0<\epsilon\leq1}$ is locally bounded, so the half-relaxed limits
\begin{align*}
 \overline V^T(t,z,b)
 &=\limsup_{\substack{\epsilon\searrow0\\ (t',z',b')\to(t,z,b)}}
 V^{\epsilon,T}(t',z',b'),\\
 \underline V^T(t,z,b)
 &=\liminf_{\substack{\epsilon\searrow0\\ (t',z',b')\to(t,z,b)}}
 V^{\epsilon,T}(t',z',b')
\end{align*}
are finite, and by the estimate of $|V^{\epsilon,T}|$ above both vanish at $t=T$. By \eqref{eq:penaltybounds} they satisfy $-\lambda\leq\underline V^T\leq\overline V^T\leq\alpha$, so both are bounded. By the standard argument \cite{CIL}, $\overline V^T$ and $\underline V^T$ are a viscosity subsolution and a viscosity supersolution of the finite horizon version of \eqref{eq:doubleVI} with zero terminal data. To compare these bounded functions on the whole state space we use $x=\log z$. The transformed generator satisfies $\A(1+x^2+b^4)\leq C(1+x^2+b^4)$, so for $\mu>C$ the function $e^{\mu(T-t)}(1+x^2+b^4)$ can be used as a barrier, and the comparison argument on compact sets extends to the whole state space \cite[Section~5.D]{CIL}. Hence $\overline V^T\leq\underline V^T$. The reverse inequality holds by definition, so $V^{\epsilon,T}$ converges locally uniformly to a continuous function $V^T$, which is the unique solution of the finite horizon obstacle problem with $V^T(T,\cdot)=0$.

Next we let $T\to\infty$. Since the coefficients and the obstacles do not depend on time, uniqueness for the finite horizon problem gives $V^T(t,\cdot)=V^{T-t}(0,\cdot)$ for $0\leq t\leq T$. In particular $V^{T+h}(T,\cdot)=V^h(0,\cdot)$ takes values in $[-\lambda,\alpha]$ for every $h>0$. Put $M=\max(\alpha,\lambda)$. The function $e^{-\delta(T-t)}$ solves $-\partial_t\varphi+\delta\varphi=0$, so by the monotonicity noted after \eqref{eq:doubleVI} the functions $V^T+Me^{-\delta(T-t)}$ and $V^T-Me^{-\delta(T-t)}$ are a supersolution and a subsolution of the finite horizon problem on $[0,T]$. At $t=T$ they lie above and below $V^{T+h}(T,\cdot)$, and comparison gives
\[
 \bigl|V^{T+h}(0,\cdot)-V^T(0,\cdot)\bigr|\leq Me^{-\delta T}.
\]
Hence $V^T(0,\cdot)$ converges uniformly as $T\to\infty$, and we denote its limit by $v_\lambda$. Since $V^T(t,\cdot)=V^{T-t}(0,\cdot)$, the functions $V^T$ also converge to $v_\lambda$ locally uniformly, and stability of viscosity solutions shows that $v_\lambda$ solves \eqref{eq:doubleVI}.

It remains to prove uniqueness and the lower bound. Let $v$ and $\widetilde v$ be bounded solutions of \eqref{eq:doubleVI} and $M_0=\|v-\widetilde v\|_\infty$. As in the limit $T\to\infty$, $\widetilde v\pm M_0e^{-\delta(T-t)}$ are a supersolution and a subsolution of the finite horizon problem which bound $v$ at $t=T$, and comparison gives $\|v-\widetilde v\|_\infty\leq e^{-\delta T}M_0$ for every $T$. Thus $v=\widetilde v$.

For the lower bound in \eqref{eq:lowzbound}, the function
\[
 w(t,z)=\alpha-\frac zr-\alpha e^{-\delta(T-t)}+\frac zre^{-r(T-t)}
\]
is affine in $z$, vanishes at $t=T$ and satisfies $w\leq\alpha$, so $\beta_\epsilon(w)\leq0$. Moreover $(-\partial_t-\A+\delta)w=1-z-A_0\leq1-z-\beta_\epsilon(w)$. By the criterion for affine functions, $w$ is a mild subsolution with terminal value zero, and $w\leq V^{\epsilon,T}$. Letting $\epsilon\searrow0$ gives $\alpha-z/r\leq V^T(0,z,b)+\alpha e^{-\delta T}$, and letting $T\to\infty$ proves the lower bound. The remaining inequalities in \eqref{eq:lowzbound} are the obstacle constraints.
\end{proof}

\subsection{Derivative representations}
We next represent the derivatives of the penalized solutions through the flow of the uncontrolled diffusion. These representations give the derivative bounds used in the boundary analysis.

On $\mathcal F_T$ define $Q^b$ by $dQ^b/dP=M_T$, and write $dW_s^{Q^b}=dW_s+(b_s/\sigma)ds$. The flow of \eqref{eq:uncontrolledZ} is explicit. Differentiating with respect to the initial factor value gives $\partial_b b_s=e^{-\kappa s}$, and differentiating the stochastic exponential for $Z_s$ gives
\[
 \partial_b\log Z_s
 =-\frac1\sigma\int_0^s e^{-\kappa u}dW_u
 -\frac1{\sigma^2}\int_0^s b_ue^{-\kappa u}du.
\]
After the substitution $dW_u=dW_u^{Q^b}-(b_u/\sigma)du$ the two finite variation terms cancel, and we obtain
\begin{equation}
 \partial_b b_s=e^{-\kappa s},\qquad
 \partial_b\log Z_s=-\frac1\sigma N_s,\qquad
 N_s=\int_0^s e^{-\kappa u}dW_u^{Q^b}.
 \label{eq:flowsensitivity}
\end{equation}

From now on we replace the piecewise linear penalty $\beta^{\rm lin}_\epsilon$ of \eqref{eq:penaltyPDE} by a smooth nondecreasing function $\beta_\epsilon$ with $0\leq\beta_\epsilon'\leq1/\epsilon$, which has the same sign as $\beta^{\rm lin}_\epsilon$ and satisfies $|\beta^{\rm lin}_\epsilon|-1\leq|\beta_\epsilon|\leq|\beta^{\rm lin}_\epsilon|$. The construction above applies without change, since the mild comparison only uses $0\leq\beta_\epsilon'\leq1/\epsilon$, and the affine barriers $-\lambda-\epsilon(1+z)$ and $\alpha+2\epsilon$ replace those in \eqref{eq:penaltybounds}, so that
\begin{equation}
 -\lambda-\epsilon(1+z)\leq V^{\epsilon,T}\leq\alpha+2\epsilon,
 \qquad -(1+z)\leq\beta_\epsilon(V^{\epsilon,T})\leq2.
 \label{eq:smoothpenaltybounds}
\end{equation}

\begin{lemma}\label{lem:linearized}
Let
\[
 k_s=\beta_\epsilon'\bigl(V^{\epsilon,T}(s,Z_s,b_s)\bigr)\geq0,
 \qquad
 R_s=\exp\left(-rs-\int_0^s k_udu\right).
\]
Then
\begin{equation}
 -(V^{\epsilon,T})_z(0,z,b)
 =\E^{Q^b}\int_0^T R_sds,
 \qquad
 (V^{\epsilon,T})_b(0,z,b)
 =\frac z\sigma\E^{Q^b}\int_0^T R_sN_sds.
 \label{eq:linearized}
\end{equation}
In particular $-1/r\leq(V^{\epsilon,T})_z<0$.
\end{lemma}

\begin{proof}
Let $Z^h$ be the solution of \eqref{eq:uncontrolledZ} started at $z+h$, and put $U_s=V^{\epsilon,T}(s,Z_s,b_s)$ and $U_s^h=V^{\epsilon,T}(s,Z_s^h,b_s)$. By the Markov property and \eqref{eq:mild}, the process $e^{-\delta s}U_s+\int_0^se^{-\delta u}(1-Z_u-\beta_\epsilon(U_u))du$ is a martingale on $[0,T]$, and the same holds for $U^h$. Let
\[
 k_s^h=\frac{\beta_\epsilon(U_s^h)-\beta_\epsilon(U_s)}{U_s^h-U_s},
\]
with $k_s^h=\beta_\epsilon'(U_s)$ when $U_s^h=U_s$, so that $0\leq k_s^h\leq1/\epsilon$. The product rule with the integrating factor $e^{-\int_0^sk_u^hdu}$ and the terminal condition $U_T^h=U_T=0$ give
\[
 \frac{U_0^h-U_0}{h}
 =-\E\int_0^T
 e^{-\delta s-\int_0^s k_u^hdu}
 \frac{Z_s^h-Z_s}{h}ds.
\]
As $h\to0$, $(Z_s^h-Z_s)/h\to\partial_zZ_s=Z_s/z$ in $L^1$ and $k_s^h\to k_s$, which gives
\[
 (V^{\epsilon,T})_z(0,z,b)
 =-\E\int_0^T
 e^{-\delta s-\int_0^s k_udu}\,\frac{Z_s}{z}ds.
\]
For the derivative in $b$ we perturb the initial factor value instead. The explicit solutions for $b_s$ and $Z_s$ give convergence of the difference quotients to $\partial_bZ_s=-Z_sN_s/\sigma$ by \eqref{eq:flowsensitivity}, and the integrability follows from
\[
 \E|\partial_bZ_s|
 =ze^{(\delta-r)s}
 \sqrt{\frac{1-e^{-2\kappa s}}{\pi\kappa\sigma^2}}.
\]
Hence $(V^{\epsilon,T})_b(0,z,b)=\E\int_0^Te^{-\delta s-\int_0^sk_udu}Z_sN_s\sigma^{-1}ds$. Finally $e^{-\delta s}Z_s=ze^{-rs}M_s$, and the change of measure $dQ^b/dP=M_T$ turns both formulas into \eqref{eq:linearized}.
\end{proof}

\subsection{A sharp occupation estimate}
The representation of $(V^{\epsilon,T})_b$ reduces the remaining estimate to a bound on the weighted occupation term involving $N_s$, and we now determine the sharp constant. Let $p_0=1+r/\kappa>1$ and define the Gaussian moment function
\begin{equation}
 \Psi(x)=\int_{\sqrt{2\kappa}x}^\infty
 (u-\sqrt{2\kappa}x)^{p_0}e^{-\frac{u^2}2}du.\label{eq:Psi}
\end{equation}
The equation $d\Psi'(d)=\Psi(d)$ has a unique solution $d_*<0$, as shown in the proof below, and we set
\begin{equation}
 c_{r,\kappa}=-d_*,\qquad L_*=\frac{c_{r,\kappa}}\sigma,
 \qquad B_*=\eta c_{r,\kappa}.\label{eq:sharpconstants}
\end{equation}

\begin{lemma}\label{lem:sharpoccupation}
Let $N_s=\int_0^s e^{-\kappa u}dW_u$, let $k$ be a nonnegative bounded adapted process, and let $R_s=e^{-rs-\int_0^s k_udu}$. Then for every $T>0$
\begin{equation}
 \left|\E\int_0^T R_sN_sds\right|
 \leq c_{r,\kappa}\E\int_0^T R_sds,
 \qquad c_{r,\kappa}\leq\frac1{\sqrt{2r}}.\label{eq:sharpoccupation}
\end{equation}
The constant $c_{r,\kappa}$ is optimal over all finite horizons and all such processes $k$.
\end{lemma}

\begin{proof}
Differentiation under the integral shows $\Psi>0$, $\Psi'<0$, $\Psi''>0$ and
\[
 \tfrac12\Psi''+\kappa x\Psi'-(r+\kappa)\Psi=0.
\]
The function $H(x)=x\Psi'(x)-\Psi(x)$ satisfies $H'(x)=x\Psi''(x)<0$ for $x<0$, $H(0)=-\Psi(0)<0$, and
$H(x)\sim\sqrt{2\pi}\,(p_0-1)(\sqrt{2\kappa}|x|)^{p_0}>0$ as $x\to-\infty$. Hence $H$ has exactly one zero $d_*$ on $(-\infty,0)$, which is the unique negative solution of $d\Psi'(d)=\Psi(d)$. Define
\begin{equation}
 \Phi(x)=\begin{cases}
 0,&x\leq d_*,\\
 \displaystyle\frac xr-\frac{d_*\Psi(x)}{r\Psi(d_*)},&x>d_*.
 \end{cases}\label{eq:Phi}
\end{equation}
The choice of $d_*$ gives $\Phi(d_*)=\Phi'(d_*)=0$, so $\Phi\in C^{1,1}_{\rm loc}$. Convexity of $\Psi$ gives $\Phi'>0$ on $(d_*,\infty)$, hence $\Phi\geq0$, and
\begin{equation}
 \tfrac12\Phi''+\kappa x\Phi'-(r+\kappa)\Phi+x\leq0,
 \label{eq:Phiineq}
\end{equation}
with equality for $x>d_*$.

Put $X_s=e^{\kappa s}(N_s-c_{r,\kappa})$, so that $X_0=d_*$ and $dX_s=\kappa X_sds+dW_s$. It\^o's formula for $e^{-(r+\kappa)s-\int_0^s k_udu}\Phi(X_s)$, together with \eqref{eq:Phiineq}, $k\Phi\geq0$ and $\Phi(X_0)=0$, gives
\[
 \E\int_0^T R_s(N_s-c_{r,\kappa})ds
 \leq
 -\E\left[
 e^{-(r+\kappa)T-\int_0^T k_udu}\Phi(X_T)
 \right]
 \leq0.
\]
The same argument applied to $-N$ gives the opposite inequality, which proves the first part of \eqref{eq:sharpoccupation}.

For the bound $c_{r,\kappa}\leq(2r)^{-1/2}$, put $c_0=(2r)^{-1/2}$ and let $\Phi_0(n)=r^{-1}(n-c_0+c_0e^{-\sqrt{2r}n})$ for $n>0$ and $\Phi_0(n)=0$ for $n\leq0$. This function is nonnegative, convex and of class $C^{1,1}$, and $\frac12a^2\Phi_0''-r\Phi_0+n-c_0\leq0$ for $0\leq a^2\leq1$. Since $dN_s=e^{-\kappa s}dW_s$ and $e^{-2\kappa s}\leq1$, It\^o's formula for $R_s\Phi_0(N_s)$, together with $k\geq0$ and $\Phi_0\geq0$, gives $\E\int_0^T R_s(N_s-c_0)ds\leq-\E[R_T\Phi_0(N_T)]\leq0$. With the same bound for $-N$, \eqref{eq:sharpoccupation} holds with $c_0$ in place of $c_{r,\kappa}$, and the optimality proved next gives $c_{r,\kappa}\leq c_0$.

To prove optimality, let $c<c_{r,\kappa}$, put $X_s=e^{\kappa s}(N_s-c)$, so that $X_0=-c>d_*$, and let $\tau$ be the first time $X$ reaches $d_*$. Since \eqref{eq:Phiineq} is an equality on $(d_*,\infty)$, It\^o's formula up to $T\wedge\tau$ gives
\[
 \E\int_0^{T\wedge\tau}e^{-rs}(N_s-c)ds
 =\Phi(-c)-\E\left[e^{-(r+\kappa)(T\wedge\tau)}\Phi(X_{T\wedge\tau})\right].
\]
The last expectation tends to zero as $T\to\infty$, because $\Phi$ grows at most linearly, $\Phi(d_*)=0$ and $r>0$. Since $\Phi(-c)>0$, there is $T$ with $\E\int_0^{T\wedge\tau}e^{-rs}(N_s-c)ds>0$. For $K>0$ the process $k^K_s=K\chi_{\{s\geq\tau\}}$ is nonnegative, bounded and adapted, and the corresponding weight is $R_s^K=e^{-rs-K(s-\tau)^+}$, which converges to $e^{-rs}\chi_{\{s\leq\tau\}}$ as $K\to\infty$. By dominated convergence,
\[
 \E\int_0^T R_s^K(N_s-c)ds
 \longrightarrow
 \E\int_0^{T\wedge\tau}e^{-rs}(N_s-c)ds>0,
\]
so $\E\int_0^T R_s^KN_sds>c\,\E\int_0^T R_s^Kds$ for large $K$, and \eqref{eq:sharpoccupation} fails with $c$ in place of $c_{r,\kappa}$.
\end{proof}

\begin{theorem}\label{thm:cone}
The solution $v_\lambda$ of \eqref{eq:doubleVI} is locally Lipschitz on $(0,\infty)\times\R$, and its weak derivatives satisfy
\begin{equation}
 -\frac1r\leq(v_\lambda)_z\leq0,
 \qquad |(v_\lambda)_b|\leq L_*z[-(v_\lambda)_z]
 \quad\text{almost everywhere}.\label{eq:cone}
\end{equation}
For all $h\in\R$,
\begin{equation}
 v_\lambda(e^{L_*|h|}z,b+h)\leq v_\lambda(z,b)
 \leq v_\lambda(e^{-L_*|h|}z,b+h).\label{eq:pointcone}
\end{equation}
The same bounds hold for the ratcheting solution of \cref{sec:ratchet}.
\end{theorem}
\begin{proof}
By \cref{lem:linearized} and \cref{lem:sharpoccupation},
\[
 -\frac1r\leq (V^{\epsilon,T})_z\leq0,
 \qquad
 |(V^{\epsilon,T})_b|
 \leq L_*z[-(V^{\epsilon,T})_z],
\]
uniformly in $\epsilon$ and $T$. Hence the functions $V^{\epsilon,T}(0,\cdot)$ are locally Lipschitz uniformly in $\epsilon$ and $T$, their limit $v_\lambda$ is locally Lipschitz, and the derivative inequalities pass to the weak derivatives, which gives \eqref{eq:cone}.

For \eqref{eq:pointcone}, put $u^{\epsilon,T}(x,b)=V^{\epsilon,T}(0,e^x,b)$, so that $|u^{\epsilon,T}_b|\leq L_*(-u^{\epsilon,T}_x)$. Along the segment from $(x,b)$ to $(x+L_*|h|,b+h)$ the derivative of $u^{\epsilon,T}$ equals $L_*|h|u^{\epsilon,T}_x+hu^{\epsilon,T}_b\leq0$, and along the segment from $(x,b)$ to $(x-L_*|h|,b+h)$ it equals $-L_*|h|u^{\epsilon,T}_x+hu^{\epsilon,T}_b\geq0$. Therefore
\[
 u^{\epsilon,T}(x+L_*|h|,b+h)\leq u^{\epsilon,T}(x,b)\leq u^{\epsilon,T}(x-L_*|h|,b+h),
\]
and letting $\epsilon\searrow0$ and $T\to\infty$ gives \eqref{eq:pointcone}. For the ratcheting problem the same argument applies with the upper penalty alone.
\end{proof}

A bound on $|v_b|$ alone would not control an intervention threshold, because $v_z$ tends to zero at contact. The Lipschitz continuity of the thresholds needs a bound on the ratio of the two derivatives. Estimate \eqref{eq:cone} gives such a bound, and it is for this purpose that the occupation weight $R_s$ is kept in \cref{lem:sharpoccupation}.

\section{Intervention thresholds and their regularity}\label{sec:regularity}

\subsection{Lipschitz thresholds}\label{sec:boundaries}
\begin{theorem}\label{thm:boundaries}
There are positive finite functions $a_\lambda,d_\lambda$ such that
\begin{align}
 \{v_\lambda=\alpha\}&=\{z\leq a_\lambda(b)\},
 &\{v_\lambda=-\lambda\}&=\{z\geq d_\lambda(b)\},\label{eq:contacts}\\
 \C_\lambda&=\{-\lambda<v_\lambda<\alpha\}
 =\{a_\lambda(b)<z<d_\lambda(b)\}.\label{eq:continuation}
\end{align}
They satisfy
\begin{equation}
 0<a_\lambda(b)\leq A_0<B_\lambda\leq d_\lambda(b)<\infty,
 \label{eq:boundarybounds}
\end{equation}
and, for $f=a_\lambda,d_\lambda$,
\begin{equation}
 |\log f(b_2)-\log f(b_1)|
 \leq L_*|b_2-b_1|.\label{eq:logLip}
\end{equation}
In particular the boundary graphs are locally Lipschitz, locally rectifiable, and of Lebesgue measure zero in the plane, and on compact factor intervals $a_\lambda$ is bounded away from zero and $d_\lambda$ is bounded above.
\end{theorem}

\begin{proof}
Write $v=v_\lambda$. The penalized solutions are nonincreasing in $z$ by \eqref{eq:linearized}, and this property passes to the limit, so $z\mapsto v(z,b)$ is nonincreasing for every $b$.

\emph{Necessary conditions for contact.} If $v(z,b)=\alpha$, the constant $\alpha$ touches $v$ from above at $(z,b)$, and the subsolution inequality for \eqref{eq:doubleVI} gives $\delta\alpha-1+z\leq0$, that is, $z\leq A_0$. Similarly, if $v(z,b)=-\lambda$, the constant $-\lambda$ touches $v$ from below, and the supersolution inequality gives $z-1-\delta\lambda\geq0$, that is, $z\geq B_\lambda$. Since $B_\lambda-A_0=\delta(\alpha+\lambda)>0$, the two contact sets are separated.

\emph{Both contact sets meet every vertical line.} Fix $b_0\in\R$ and a bounded interval $I$ containing $b_0$ in its interior. For $\varepsilon>0$ let $\psi_\varepsilon$ be a smooth function on $(0,\varepsilon)\times I$ which depends on $z$ through $z/\varepsilon$, takes values in $[0,1]$, equals one near $z=\varepsilon$ and near the lateral sides of $I$, and vanishes in a neighborhood of $(\varepsilon/4,b_0)$. For $0<\zeta<\varepsilon/4$ put
\[
 w_\zeta(z,b)=\alpha-\frac zr\psi_\varepsilon(z,b)-\frac{\zeta}{r}.
\]
By \eqref{eq:lowzbound}, $w_\zeta\leq\alpha-\zeta/r\leq v$ on $\{z=\zeta\}$, and $w_\zeta\leq\alpha-z/r\leq v$ on $\{z=\varepsilon\}$ and on the lateral boundary. The scaling of $\psi_\varepsilon$ gives $|(\A-\delta)[z\psi_\varepsilon]|\leq C_I\varepsilon$ on $(0,\varepsilon)\times I$, hence $(\A-\delta)w_\zeta+1-z\geq A_0-C_I'\varepsilon>0$ for small $\varepsilon$. Together with $w_\zeta\leq\alpha$, this shows that $w_\zeta$ is a subsolution of \eqref{eq:doubleVI} in $(\zeta,\varepsilon)\times I$, and comparison gives $w_\zeta\leq v$ there. Where $\psi_\varepsilon=0$ we obtain $v\geq\alpha-\zeta/r$, and letting $\zeta\searrow0$ shows that $v=\alpha$ in a neighborhood of $(\varepsilon/4,b_0)$.

The lower contact set is treated in the same way for large $z$. For $R>0$ let $\psi_R$ be a smooth function on $(R/2,2R)\times I$ which depends on $z$ through $z/R$, takes values in $[0,1]$, equals one near the boundary of this rectangle and vanishes near $(R,b_0)$, and put $\overline w=-\lambda+(\alpha+\lambda)\psi_R$. Then $\overline w=\alpha\geq v$ on the boundary of the rectangle and
\[
 (\A-\delta)\overline w+1-z
 \leq B_\lambda+C_I(\alpha+\lambda)-\frac R2<0
\]
for large $R$. Since $\overline w\geq-\lambda$, the function $\overline w$ is a supersolution, comparison gives $v\leq\overline w$, and $v=-\lambda$ where $\psi_R=0$.

\emph{The thresholds.} We may therefore define
\[
 a_\lambda(b)=\sup\{z>0:v(z,b)=\alpha\},\qquad
 d_\lambda(b)=\inf\{z>0:v(z,b)=-\lambda\}.
\]
The contact conditions give $a_\lambda\leq A_0$ and $d_\lambda\geq B_\lambda$, and the two barriers give $a_\lambda>0$ and $d_\lambda<\infty$. Since $v(\cdot,b)$ is continuous and nonincreasing, \eqref{eq:contacts} holds, and together with $A_0<B_\lambda$ this gives \eqref{eq:continuation}. The barriers can be used for all $b$ in a neighborhood of $b_0$, and a finite covering shows that $a_\lambda$ is bounded away from zero and $d_\lambda$ is bounded above on compact intervals.

\emph{The Lipschitz bound.} Fix $b_1,b_2$ and put $\Delta=L_*|b_2-b_1|$. Since $v(a_\lambda(b_1),b_1)=\alpha$, the second inequality in \eqref{eq:pointcone} gives $v(e^{-\Delta}a_\lambda(b_1),b_2)\geq\alpha$, so $e^{-\Delta}a_\lambda(b_1)\leq a_\lambda(b_2)$. Interchanging $b_1$ and $b_2$ yields $e^{-\Delta}a_\lambda(b_1)\leq a_\lambda(b_2)\leq e^{\Delta}a_\lambda(b_1)$, which is \eqref{eq:logLip} for $a_\lambda$. In the same way, $v(d_\lambda(b_1),b_1)=-\lambda$ and the first inequality in \eqref{eq:pointcone} give $v(e^{\Delta}d_\lambda(b_1),b_2)\leq-\lambda$, so $d_\lambda(b_2)\leq e^\Delta d_\lambda(b_1)$, and interchanging $b_1$ and $b_2$ proves \eqref{eq:logLip} for $d_\lambda$. The remaining assertions follow from \eqref{eq:logLip} and the local bounds.
\end{proof}

\subsection{Interior regularity and comparison in the costs}
In $x=\log z$, the diffusion vector field and the Stratonovich drift are
\[
 V_1=\Bigl(-\frac b\sigma,-\eta\Bigr),\qquad
 V_0=\Bigl(\delta-r-\frac{b^2}{2\sigma^2}-\frac\eta{2\sigma},\ \kappa(\bar b-b)\Bigr).
\]
With $[U,V]=DV\,U-DU\,V$,
\begin{equation}
 [V_1,[V_0,V_1]]=
 \left(-\frac{\eta(2\kappa\sigma-\eta)}{\sigma^2},0\right),
 \qquad
 \det(V_1,[V_1,[V_0,V_1]])
 =-\frac{\eta^2(2\kappa\sigma-\eta)}{\sigma^2}.
 \label{eq:bracket}
\end{equation}
Since $2\kappa\sigma>\eta$, the determinant does not vanish, and H\"ormander's condition holds at every point. On a compact subset of $\C_\lambda$ the stationary penalized solutions of the proof of \cref{thm:intrinsic} converge uniformly to $v_\lambda$, so for small $\epsilon$ they take values in $(-\lambda,\alpha)$ and their penalty vanishes. Passing to the limit in the distributional form of the penalized equation, $(\delta-\A)v_\lambda=1-z$ holds in distributions in $\C_\lambda$, and H\"ormander's theorem \cite{Hormander} gives $v_\lambda\in C^\infty(\C_\lambda)$.

Differentiating the equation in $z$ gives
\begin{equation}
 (\A_1-r)(v_\lambda)_z=1,\qquad
 \A_1=\A+\frac{b^2z}{\sigma^2}\partial_z+
              \frac{\eta b}{\sigma}\partial_b.\label{eq:derivativePDE}
\end{equation}
Let $w=(v_\lambda)_z$, so that $w\leq0$ by \eqref{eq:cone}. If $w(z_0,b_0)=0$ at a point of $\C_\lambda$, then $w$ has a local maximum there, its first derivatives vanish and the second order part of $\A_1w$ is nonpositive, so $(\A_1-r)w(z_0,b_0)\leq0$, contradicting \eqref{eq:derivativePDE}. Therefore
\begin{equation}
 -\frac1r\leq(v_\lambda)_z<0
 \quad\text{in }\C_\lambda.\label{eq:strictderiv}
\end{equation}
By \eqref{eq:contacts}, \eqref{eq:continuation} and \eqref{eq:strictderiv}, for every $\ell\in(-\lambda,\alpha)$ and $b\in\R$ the equation $v_\lambda(z,b)=\ell$ has exactly one solution $z$, and the implicit function theorem shows that the level set $\{v_\lambda=\ell\}$ is a smooth graph over $b$.

\begin{proposition}\label{prop:costs}
For $\lambda_2\geq\lambda_1$,
\begin{equation}
 -(\lambda_2-\lambda_1)\leq v_{\lambda_2}-v_{\lambda_1}\leq0,
 \quad a_{\lambda_2}\leq a_{\lambda_1},\quad
 d_{\lambda_2}\geq d_{\lambda_1}.\label{eq:costmonotone}
\end{equation}
The marginal equation and its boundaries do not depend on $\gamma$.
\end{proposition}
\begin{proof}
Since $v_{\lambda_2}+\lambda_1\leq v_{\lambda_2}+\lambda_2$, the monotonicity of \eqref{eq:doubleVI} shows that $v_{\lambda_2}$ is a subsolution of the problem with cost $\lambda_1$, and comparison gives $v_{\lambda_2}\leq v_{\lambda_1}$. Put $\Delta=\lambda_2-\lambda_1$ and $w=v_{\lambda_1}-\Delta$. Then $w+\lambda_2=v_{\lambda_1}+\lambda_1$, while $\delta w-\A w-1+z$ and $w-\alpha$ are the corresponding expressions for $v_{\lambda_1}$ decreased by $\delta\Delta$ and $\Delta$. Hence $w$ is a subsolution of the problem with cost $\lambda_2$, and $v_{\lambda_1}-\Delta\leq v_{\lambda_2}$. This proves the first two inequalities in \eqref{eq:costmonotone}.

If $v_{\lambda_2}(z,b)=\alpha$, then $\alpha=v_{\lambda_2}(z,b)\leq v_{\lambda_1}(z,b)\leq\alpha$, so the upper contact set for $\lambda_2$ is contained in that for $\lambda_1$, and $a_{\lambda_2}\leq a_{\lambda_1}$. The lower estimate gives $v_{\lambda_2}+\lambda_2\geq v_{\lambda_1}+\lambda_1\geq0$, so $v_{\lambda_2}(z,b)=-\lambda_2$ implies $v_{\lambda_1}(z,b)=-\lambda_1$, and $d_{\lambda_2}\geq d_{\lambda_1}$. Finally, $\gamma$ does not appear in \eqref{eq:doubleVI}.
\end{proof}

\begin{proposition}\label{prop:reduction}
Write $v_{\alpha,\lambda}$ for the solution of \eqref{eq:doubleVI} with costs $(\alpha,\lambda)$, and $a_{\alpha,\lambda}$ and $d_{\alpha,\lambda}$ for its thresholds in \cref{thm:boundaries}. With $\hat\lambda=(\alpha+\lambda)/A_0$,
\begin{equation}
 v_{\alpha,\lambda}(z,b)
 =\alpha+A_0\,v_{0,\hat\lambda}\Bigl(\frac z{A_0},b\Bigr),\qquad
 a_{\alpha,\lambda}=A_0\,a_{0,\hat\lambda},\qquad
 d_{\alpha,\lambda}=A_0\,d_{0,\hat\lambda}.
 \label{eq:reduction}
\end{equation}
In particular
$d_{\alpha,\lambda}/a_{\alpha,\lambda}\geq B_\lambda/A_0=\mathrm{ILA}$, so the ratio of the two thresholds depends on the costs only through the ILA coefficient, and their common scale is $A_0$.
\end{proposition}

\begin{proof}
Put $\zeta=z/A_0$ and $w(z,b)=\alpha+A_0v_{0,\hat\lambda}(\zeta,b)$. The operator $\A$ commutes with the dilation $z\mapsto z/A_0$, and $\delta\alpha-1=-A_0$, so
\[
 \delta w-\A w-1+z
 =A_0\bigl[\delta v_{0,\hat\lambda}-\A v_{0,\hat\lambda}-1+\zeta\bigr](\zeta,b),
\]
while $w-\alpha=A_0v_{0,\hat\lambda}(\zeta,b)$ and $w+\lambda=A_0(v_{0,\hat\lambda}(\zeta,b)+\hat\lambda)$. Thus the three expressions in \eqref{eq:doubleVI} for $w$ with costs $(\alpha,\lambda)$ are $A_0$ times those for $v_{0,\hat\lambda}$ with costs $(0,\hat\lambda)$, and uniqueness in \cref{prop:PDEexist} gives $w=v_{\alpha,\lambda}$. The identities for the thresholds follow from \eqref{eq:contacts}. Finally, \eqref{eq:boundarybounds} for the costs $(0,\hat\lambda)$ gives $a_{0,\hat\lambda}\leq1$ and $d_{0,\hat\lambda}\geq1+\delta\hat\lambda=B_\lambda/A_0$, so $d_{\alpha,\lambda}/a_{\alpha,\lambda}=d_{0,\hat\lambda}/a_{0,\hat\lambda}\geq B_\lambda/A_0$.
\end{proof}

\subsection{Straightening the diffusion}
We now rewrite the obstacle problem in coordinates adapted to the diffusion direction. Define
\begin{equation}
 s=\log z-\frac{b^2}{2\sigma\eta},\qquad
 U_\lambda(s,b)=v_\lambda\Bigl(e^{s+\frac{b^2}{2\sigma\eta}},b\Bigr).
 \label{eq:straighten}
\end{equation}
Since $d\log Z_t=(\delta-r-b_t^2/(2\sigma^2))dt-(b_t/\sigma)dW_t$, It\^o's formula gives
\[
 ds_t
 =d\log Z_t-\frac{b_t}{\sigma\eta}db_t-\frac{\eta}{2\sigma}dt
 =G(b_t)dt,
\]
so the new coordinate has no martingale part. Here
\begin{align}
 G(b)&=\delta-r-\frac{\eta}{2\sigma}
 +\frac{2\kappa\sigma-\eta}{2\sigma^2\eta}b^2
 -\frac{\kappa\bar b}{\sigma\eta}b,\label{eq:G}\\
 \mathcal K&=\frac{\eta^2}{2}\partial_{bb}
 +\mu(b)\partial_b+G(b)\partial_s,
 \qquad\mu(b)=\kappa(\bar b-b),\label{eq:K}
\end{align}
and $\mathcal K$ is the generator in the coordinates $(s,b)$.

\begin{theorem}\label{thm:intrinsic}
For every compact rectangle $E$ in the coordinates $(s,b)$ there is a constant $C_E$, independent of $\lambda$, such that
\begin{equation}
 \|U_\lambda\|_{W^{1,\infty}(E)}+
 \|(U_\lambda)_{bb}\|_{L^\infty(E)}\leq C_E.\label{eq:intrinsic}
\end{equation}
The derivative $(U_\lambda)_b$ has a continuous version satisfying
\begin{equation}
 |(U_\lambda)_b(s,b)-(U_\lambda)_b(s',b')|
 \leq C_E\bigl(|b-b'|+|s-s'|^{\frac12}\bigr),\label{eq:Ubholder}
\end{equation}
and $(U_\lambda)_b=0$ at every contact point. Equivalently, the derivative of $v_\lambda$ in the diffusion direction,
\begin{equation}
 \mathcal Dv_\lambda=\frac b\sigma z(v_\lambda)_z+\eta(v_\lambda)_b
 =\eta(U_\lambda)_b,\label{eq:diffsmoothfit}
\end{equation}
is continuous and vanishes on both free boundaries, and $\mathcal D^2v_\lambda=\eta^2(U_\lambda)_{bb}\in L^\infty_{\rm loc}$. If $(s_0,b_0)$ is an upper contact point and $\phi=\alpha-U_\lambda$, or a lower contact point and $\phi=U_\lambda+\lambda$, then
\begin{equation}
 0\leq \phi(s,b)\leq C_E\bigl(|s-s_0|+|b-b_0|^2\bigr)
 \label{eq:anisogrowth}
\end{equation}
in a neighborhood of $(s_0,b_0)$. The same conclusions hold for the ratcheting problem.
\end{theorem}

\begin{proof}
Fix $0<\epsilon\leq1$. We first let $T\to\infty$ in $V^{\epsilon,T}$, working in the space $\mathcal C_1$ of the proof of \cref{prop:PDEexist}, since the penalized solutions grow linearly in $z$. By uniqueness of mild solutions and time homogeneity, $V^{\epsilon,T+h}$ restricted to $[0,T]$ is the mild solution with terminal value $\varphi_h=V^{\epsilon,h}(0,\cdot)$, and by \eqref{eq:smoothpenaltybounds} and the lower bound at the end of that proof, $-z/r\leq\varphi_h\leq\alpha+2$, so $|\varphi_h|\leq C_0(1+z)$ with $C_0=\alpha+2+1/r$, which also covers the ratcheting problem. Put $\mathcal W(t)=e^{-\delta(T-t)}\mathcal S_{T-t}|\varphi_h|$. Since $\beta_\epsilon$ is nondecreasing and $\mathcal W\geq0$, \eqref{eq:Kshift} shows that $V^{\epsilon,T}+\mathcal W$ is a mild supersolution with terminal value $|\varphi_h|$ and $V^{\epsilon,T}-\mathcal W$ is a mild subsolution with terminal value $-|\varphi_h|$. The mild comparison in that proof gives
\[
 \bigl|V^{\epsilon,T+h}(0,\cdot)-V^{\epsilon,T}(0,\cdot)\bigr|
 \leq e^{-\delta T}\mathcal S_T|\varphi_h|
 \leq C_0\bigl(e^{-\delta T}+ze^{-rT}\bigr).
\]
Hence $V^{\epsilon,T}(0,\cdot)$ is a Cauchy family in $\mathcal C_1$ as $T\to\infty$, and its limit $v^\epsilon$ is a stationary penalized solution, in the mild sense and, by stability, in the viscosity sense. The bounds \eqref{eq:smoothpenaltybounds} pass to the limit. The subsolution $w$ at the end of the same proof gives $V^{\epsilon,T}(0,\cdot)\geq\alpha-z/r-\alpha e^{-\delta T}+ze^{-rT}/r$, so $\alpha-z/r\leq v^\epsilon\leq\alpha+2\epsilon$. Hence $v^\epsilon$ and $\beta_\epsilon(v^\epsilon)$ are locally bounded independently of $\lambda$ and $\epsilon$.

Put $U^\epsilon(s,b)=v^\epsilon(e^{s+b^2/(2\sigma\eta)},b)$. In distributions, the stationary penalized equation becomes
\begin{equation}
 \frac{\eta^2}{2}U^\epsilon_{bb}
 =\delta U^\epsilon-\mu U^\epsilon_b-GU^\epsilon_s
   +z-1+\beta_\epsilon(U^\epsilon),\qquad
 z=e^{s+\frac{b^2}{2\sigma\eta}}.\label{eq:Ubbpenalty}
\end{equation}
Letting $T\to\infty$ in the bounds of \cref{lem:linearized,lem:sharpoccupation} gives $-1/r\leq(v^\epsilon)_z\leq0$ and $|(v^\epsilon)_b|\leq L_*z[-(v^\epsilon)_z]$. Since
\[
 U^\epsilon_s=z(v^\epsilon)_z,\qquad
 U^\epsilon_b=(v^\epsilon)_b+\frac{b}{\sigma\eta}z(v^\epsilon)_z,
\]
we obtain $|U^\epsilon_s|\leq z/r$ and $|U^\epsilon_b|\leq (z/r)(L_*+|b|/(\sigma\eta))$, and \eqref{eq:Ubbpenalty} then bounds $U^\epsilon_{bb}$. Thus \eqref{eq:intrinsic} holds for $U^\epsilon$ with $C_E$ independent of $\lambda$ and $\epsilon$. These bounds give local compactness as $\epsilon\searrow0$. Every locally uniform limit solves \eqref{eq:doubleVI}, so uniqueness identifies it with $v_\lambda$, and weak-$*$ convergence of the first derivatives and of $U^\epsilon_{bb}$ gives \eqref{eq:intrinsic}.

For \eqref{eq:Ubholder}, choose a slightly larger rectangle on which \eqref{eq:intrinsic} holds. For almost every $s$, the function $b\mapsto(U_\lambda)_b(s,b)$ is Lipschitz, so
\[
 \left|(U_\lambda)_b(s,b)-\frac{U_\lambda(s,b+h)-U_\lambda(s,b)}{h}\right|
 \leq C_E|h|,
\]
while the bound on $(U_\lambda)_s$ shows that the difference quotients at $s$ and $s'$ differ by at most $2C_E|s-s'|/|h|$. Taking $|h|=|s-s'|^{1/2}$ gives $|(U_\lambda)_b(s,b)-(U_\lambda)_b(s',b)|\leq C_E|s-s'|^{1/2}$ for almost every $s,s'$. With the Lipschitz bound in $b$, this proves \eqref{eq:Ubholder} on a set of full measure and defines a continuous extension of $(U_\lambda)_b$. The identity $U_\lambda(s,b+h)-U_\lambda(s,b)=\int_b^{b+h}(U_\lambda)_b(s,q)dq$ holds for almost every $s$ and then for every $s$ by continuity, so the extension is the classical derivative in $b$.

At an upper contact point $(s_0,b_0)$ the function $b\mapsto U_\lambda(s_0,b)$ attains its maximum $\alpha$ at $b_0$, so $(U_\lambda)_b(s_0,b_0)=0$, and at a lower contact point it attains its minimum $-\lambda$. The change of variables \eqref{eq:straighten} turns $\mathcal D$ into $\eta\partial_b$, which gives \eqref{eq:diffsmoothfit} and the bound on $\mathcal D^2v_\lambda$. Finally, at a contact point $\phi(s_0,b_0)=\phi_b(s_0,b_0)=0$, and the Lipschitz continuity of $\phi_b$ in $b$ gives $0\leq\phi(s_0,b)\leq C_E|b-b_0|^2$. The bound on $U_s$ gives $|\phi(s,b)-\phi(s_0,b)|\leq C_E|s-s_0|$, and \eqref{eq:anisogrowth} follows. For the ratcheting problem the same argument applies to the penalized problem without the lower obstacle.
\end{proof}

The bound on $(U_\lambda)_{bb}$ controls only the second derivative along the diffusion. The derivatives $v_{zz}$, $v_{zb}$ and $v_{bb}$ in the original coordinates need not be locally bounded, and the difference matters at contact points where the diffusion is tangent to the boundary (\cref{sec:characteristic}).

\subsection{Smooth free boundaries}\label{sec:smooth}
Put
\begin{equation}
 \mathcal T=\{b\in\R: |b|>B_*,\ G(b)\ne0\},
 \qquad B_*=\eta c_{r,\kappa}.\label{eq:smoothset}
\end{equation}
Since $c_{r,\kappa}\leq(2r)^{-1/2}$, the set $\mathcal T$ contains $\{|b|>\eta/\sqrt{2r},\ G(b)\neq0\}$.

\begin{theorem}\label{thm:smoothtails}
Let $0\leq\lambda<\infty$. Then
\begin{equation}
 a_\lambda,d_\lambda\in C^\infty(\mathcal T),\qquad
 a_\lambda(b)<A_0,\quad d_\lambda(b)>B_\lambda
 \quad(b\in\mathcal T).\label{eq:tailconclusions}
\end{equation}
The ratcheting threshold $a_\infty$ also belongs to $C^\infty(\mathcal T)$ and satisfies $a_\infty<A_0$ on $\mathcal T$. Moreover, $v_\lambda$ is smooth up to each free boundary over $\mathcal T$ from the continuation region, $(v_\lambda)_z=(v_\lambda)_b=0$ there, and the obstacle gap grows quadratically with the distance to the free boundary.
\end{theorem}

The proof occupies the rest of this subsection and the next two. We first write each free boundary over $\mathcal T$ locally as a Lipschitz graph for a uniformly parabolic obstacle problem in one space variable.

\paragraph{A parabolic chart}\label{sec:chart}
For a threshold $f\in\{a_\lambda,d_\lambda,a_\infty\}$ put
\begin{equation}
 F_f(b)=\log f(b),\qquad
 h_f(b)=F_f(b)-\frac{b^2}{2\sigma\eta},
 \label{eq:adaptedboundary}
\end{equation}
so that the boundary is $\{s=h_f(b)\}$ in adapted coordinates. By \eqref{eq:logLip}, almost everywhere
\begin{equation}
 -L_*-\frac b{\sigma\eta}\leq h_f'(b)
 \leq L_*-\frac b{\sigma\eta}.\label{eq:hslope}
\end{equation}
Since $L_*=B_*/(\sigma\eta)$, this gives $h_f'(b)\leq-(b-B_*)/(\sigma\eta)<0$ for $b>B_*$ and $h_f'(b)\geq(|b|-B_*)/(\sigma\eta)>0$ for $b<-B_*$. Hence $h_f$ is strictly decreasing on $(B_*,\infty)$ and strictly increasing on $(-\infty,-B_*)$, and on every compact interval $I\subset\{|b|>B_*\}$ there is $c_I>0$ with $|h_f(b_2)-h_f(b_1)|\geq c_I|b_2-b_1|$ for $b_1,b_2\in I$. Thus $h_f$ is invertible on $I$ with a Lipschitz inverse.

Fix a compact interval $I\subset\mathcal T$. Then $G$ has a constant sign on $I$, and we put $\tau=-\sgn(G)\,s$. Since $\partial_s=-\sgn(G)\partial_\tau$,
\begin{equation}
 \frac{\mathcal K-\delta}{|G|}=-\mathcal P,\qquad
 \mathcal P=\partial_\tau-\widetilde a(b)\partial_{bb}
 -\widetilde m(b)\partial_b+\widetilde c(b),
 \label{eq:parabolicchart}
\end{equation}
with $\widetilde a=\eta^2/(2|G|)$, $\widetilde m=\mu/|G|$ and $\widetilde c=\delta/|G|$. Because $G$ is bounded away from zero on $I$, these coefficients are smooth and bounded on $I$, and $0<c_I\leq\widetilde a\leq C_I$, so $\mathcal P$ is uniformly parabolic with $\tau$ as time. In these coordinates the free boundary is $\{\tau=-\sgn(G)h_f(b)\}$, and since $h_f$ has a Lipschitz inverse on $I$, it can be written as $\{b=\Gamma(\tau)\}$ with $\Gamma$ Lipschitz. The two free boundaries are separated by \eqref{eq:boundarybounds}, so near each of them the continuation region lies on one side of this graph.

Since $b$ is the only space variable, the parabolic $C^1$ condition of \cite[Definitions~1.1 and~2.2]{Kukuljan} for the domain $\{b>\Gamma(\tau)\}$ reduces to $|\Gamma(\tau)-\Gamma(\tau')|\leq|\tau-\tau'|^{1/2}\omega(|\tau-\tau'|^{1/2})$ for a modulus of continuity $\omega$. A Lipschitz $\Gamma$ satisfies this with $\omega(\rho)=C\rho$. After replacing $b$ by $-b$ if the continuation region lies on the other side, the continuation region is therefore a $C_p^1$ domain near every free boundary point over $I$.

\subsection{Full smooth fit and nondegeneracy}\label{sec:fullfit}
We next use the cone estimate and the regularity in adapted coordinates to show that both first derivatives vanish at the free boundary over $\mathcal T$, and we derive a parabolic equation for $-U_s$ which gives the nondegeneracy needed for the boundary Harnack argument.

Fix a compact interval $I\subset\mathcal T$ and write $U=U_\lambda$. In the continuation region $U_s=z(v_\lambda)_z<0$ by \eqref{eq:strictderiv}, and since $U_b=(v_\lambda)_b+(b/(\sigma\eta))U_s$, the estimate \eqref{eq:cone} gives
\begin{equation}
 -\left(\frac b{\sigma\eta}+L_*\right)(-U_s)
 \leq U_b\leq
 -\left(\frac b{\sigma\eta}-L_*\right)(-U_s).
 \label{eq:tailderivatives}
\end{equation}
Since $B_*=\sigma\eta L_*$, it follows that $U_b<0$ for $b>B_*$, $U_b>0$ for $b<-B_*$, and
\begin{equation}
 \frac{|b|-B_*}{\sigma\eta}(-U_s)
 \leq |U_b|\leq
 \frac{|b|+B_*}{\sigma\eta}(-U_s).
 \label{eq:twosided}
\end{equation}
By \cref{thm:intrinsic}, $U_b$ is continuous and vanishes on the free boundary. As $I$ is a compact subset of $\{|b|>B_*\}$, \eqref{eq:twosided} gives $0\leq-U_s\leq c_I^{-1}|U_b|$, so $U_s$ also tends to zero at the free boundary. Since $(v_\lambda)_z=U_s/z$ and $(v_\lambda)_b=U_b-(b/(\sigma\eta))U_s$, both $(v_\lambda)_z$ and $(v_\lambda)_b$ extend continuously by zero to the free boundary over $I$.

In the continuation region $(\mathcal K-\delta)U=e^{s+b^2/(2\sigma\eta)}-1$. The coefficients of $\mathcal K$ do not depend on $s$, so differentiation in $s$ gives $(\mathcal K-\delta)(-U_s)=-e^{s+b^2/(2\sigma\eta)}$, and by \eqref{eq:parabolicchart}
\begin{equation}
 \mathcal P(-U_s)
 =\frac{e^{s+\frac{b^2}{2\sigma\eta}}}{|G(b)|}>0,
 \qquad
 U_s=0\quad\text{on }\{b=\Gamma(\tau)\}.
 \label{eq:qparabolic}
\end{equation}
Fix one of the free boundary portions over $I$, with graph $b=\Gamma(\tau)$, and let $\phi=\alpha-U$ near $a_\lambda$ and $\phi=U+\lambda$ near $d_\lambda$. Then $\phi=0$ on the free boundary and $\phi>0$ in the continuation region.

\begin{lemma}\label{lem:qhopf}
Let $J$ be a compact interval in the domain of $\Gamma$. There are constants $c,C,\varrho>0$ such that for every continuation point $(\tau,b)$ with $\tau\in J$ and $\Delta=|b-\Gamma(\tau)|<\varrho$,
\begin{equation}
 -U_s(\tau,b)\geq c\Delta,
 \label{eq:qnondeg}
\end{equation}
and
\begin{equation}
 c\Delta\leq|\phi_b(\tau,b)|\leq C\Delta,
 \qquad
 |\phi_\tau(\tau,b)|\leq C\Delta.
 \label{eq:gapderivativegrowth}
\end{equation}
\end{lemma}

\begin{proof}
Fix $\tau_0\in J$ and put $b_0=\Gamma(\tau_0)$. We treat the case in which the continuation region is locally $\{b>\Gamma(\tau)\}$, the other case being symmetric. Let $L_\Gamma$ be a Lipschitz constant of $\Gamma$ on a slightly larger interval and fix $M>L_\Gamma$. For small $h,\ell>0$ put $y(\tau,b)=b-b_0+M(\tau-\tau_0)$ and
\[
 S=\left\{(\tau,b):\tau_0-h<\tau<\tau_0,\ 0<y(\tau,b)<\ell\right\}.
\]
Since $b-\Gamma(\tau)\geq y(\tau,b)+(M-L_\Gamma)(\tau_0-\tau)$, the closure of $S$ without the point $(\tau_0,b_0)$ lies in the continuation region. Let
\[
 w(\tau,b)=\varepsilon(\tau-\tau_0+h)\,y(\tau,b)\bigl(\ell-y(\tau,b)\bigr).
\]
A direct calculation gives
\[
 \mathcal Pw
 =\varepsilon\Bigl[y(\ell-y)+(\tau-\tau_0+h)\bigl(M(\ell-2y)+2\widetilde a-\widetilde m(\ell-2y)+\widetilde c\,y(\ell-y)\bigr)\Bigr]
 \leq\varepsilon C_0
\]
in $S$, with $C_0$ uniform in $\tau_0\in J$. On the same neighborhood \eqref{eq:qparabolic} gives $\mathcal P(-U_s)\geq q_0>0$, and we choose $\varepsilon$ with $\varepsilon C_0\leq q_0$. Then $\mathcal Pw\leq\mathcal P(-U_s)$ in $S$. The function $w$ vanishes on the bottom and on both lateral sides of $S$, where $-U_s\geq0$, so the parabolic comparison principle gives $w\leq-U_s$ in $S$. At $\tau=\tau_0$ and $0<b-b_0<\ell/2$ this yields $-U_s(\tau_0,b)\geq\frac12\varepsilon h\ell\,|b-b_0|$, and uniformity in $\tau_0\in J$ proves \eqref{eq:qnondeg}.

For either choice of $\phi$ we have $|\phi_b|=|U_b|$, so \eqref{eq:twosided} and \eqref{eq:qnondeg} give $|\phi_b|\geq c\Delta$. Since $U_b=0$ on the free boundary and $U_{bb}\in L^\infty_{\rm loc}$, writing $U_b(\tau,b)=\int_{\Gamma(\tau)}^bU_{bb}(\tau,q)dq$ gives $|\phi_b|\leq C\Delta$. Finally $|\phi_\tau|=|U_s|\leq C|U_b|\leq C\Delta$ by \eqref{eq:twosided}.
\end{proof}

\paragraph{Strict forcing}\label{sec:strict}
In the continuation region,
\begin{align}
 (\mathcal K-\delta)(\alpha-U)
 &=A_0-e^{s+\frac{b^2}{2\sigma\eta}},
 \label{eq:uppergap}\\
 (\mathcal K-\delta)(U+\lambda)
 &=e^{s+\frac{b^2}{2\sigma\eta}}-B_\lambda.
 \label{eq:lowergap}
\end{align}
By \eqref{eq:parabolicchart}, $\mathcal P\phi=(e^{s+b^2/(2\sigma\eta)}-A_0)/|G|$ for $\phi=\alpha-U$ and $\mathcal P\phi=(B_\lambda-e^{s+b^2/(2\sigma\eta)})/|G|$ for $\phi=U+\lambda$. Writing
\[
 \widetilde a\phi_{bb}=\phi_\tau-\widetilde m\phi_b+\widetilde c\phi-\mathcal P\phi
\]
and using $\phi\to0$ together with \eqref{eq:gapderivativegrowth}, we see that $\phi_{bb}$ extends continuously from the continuation region to the free boundary, with
\begin{equation}
 \phi_{bb}\big|_{b=\Gamma(\tau)}
 =
 \begin{cases}
 \dfrac{2(A_0-a_\lambda(b))}{\eta^2},
 & \phi=\alpha-U,\\[8pt]
 \dfrac{2(d_\lambda(b)-B_\lambda)}{\eta^2},
 & \phi=U+\lambda.
 \end{cases}
 \label{eq:adaptedcurvature}
\end{equation}
For fixed $\tau$, since $\phi_b(\tau,\Gamma(\tau))=0$, the ratio $|\phi_b(\tau,b)|/|b-\Gamma(\tau)|$ converges to the absolute value of this trace as $b\to\Gamma(\tau)$, and the lower bound in \eqref{eq:gapderivativegrowth} shows that the trace is at least $c>0$. By \eqref{eq:boundarybounds} the trace is nonnegative, so it is positive, which gives $a_\lambda(b)<A_0$ and $d_\lambda(b)>B_\lambda$. Moreover, by \eqref{eq:gapderivativegrowth} the derivative $\phi_b$ does not change sign between the free boundary and a nearby continuation point, so $\phi(\tau,b)=|\int_{\Gamma(\tau)}^b\phi_b(\tau,q)dq|$, and integration of \eqref{eq:gapderivativegrowth} gives
\begin{equation}
 c\Delta^2\leq\phi(\tau,b)\leq C\Delta^2,
 \qquad
 \Delta=|b-\Gamma(\tau)|.
 \label{eq:quadraticcontact}
\end{equation}

\subsection{Boundary Harnack estimates and higher regularity}\label{sec:bh}
\paragraph{The first derivative of the graph}
Fix a free boundary point $P_0=(\tau_0,b_0)$ with $b_0=\Gamma(\tau_0)$, let $Q_r(P_0)=(\tau_0-r^2,\tau_0]\times(b_0-r,b_0+r)$, and let $\Omega_r=Q_r(P_0)\cap\{\phi>0\}$. For small $r$, $\Omega_r$ is a $C_p^1$ domain in the sense of \cite[Definitions~1.1 and~2.2]{Kukuljan}, and its free boundary in $Q_r(P_0)$ is $\{b=\Gamma(\tau)\}$. In $\Omega_r$,
\begin{equation}
 \mathcal P\phi=-\widetilde f,\qquad
 \widetilde f(\tau,b)=
 \begin{cases}
 \dfrac{A_0-e^{-\sgn(G)\tau+\frac{b^2}{2\sigma\eta}}}{|G(b)|},
 & \phi=\alpha-U,\\[4pt]
 \dfrac{e^{-\sgn(G)\tau+\frac{b^2}{2\sigma\eta}}-B_\lambda}{|G(b)|},
 & \phi=U+\lambda,
 \end{cases}
 \label{eq:gappde}
\end{equation}
and $\widetilde f$ is smooth. Since $\phi$ is smooth in the continuation region, differentiation of \eqref{eq:gappde} gives
\begin{equation}
 \mathcal P\phi_\tau=-\widetilde f_\tau,
 \qquad
 \mathcal P\phi_b
 =\widetilde a'\phi_{bb}+\widetilde m'\phi_b-\widetilde c'\phi-\widetilde f_b.
 \label{eq:derivativepair}
\end{equation}
After the zeroth order terms are moved to the right, $\phi_\tau$ and $\phi_b$ therefore solve $\partial_\tau u-\widetilde a\partial_{bb}u-\widetilde m\partial_bu=f$ with bounded right sides, since $\phi_{bb}=\mp U_{bb}$ is bounded. By \eqref{eq:gapderivativegrowth}, $\phi_\tau=\phi_b=0$ on $\{b=\Gamma(\tau)\}$ and $|\phi_b(\tau,b)|\geq c|b-\Gamma(\tau)|$. Taking $r$ small, $\phi_b$ has a fixed sign in $\Omega_{2r}$, and after replacing $\phi_b$ by $-\phi_b$ if necessary we may assume $\phi_b>0$ there. Then
\[
 \phi_b(\tau,b)\geq c|b-\Gamma(\tau)|\geq c\operatorname{dist}\bigl(b,\partial\Omega_{2r,\tau}\bigr)
\]
in $\Omega_{2r}$, where $\Omega_{2r,\tau}$ is the time slice of $\Omega_{2r}$. The coefficients are smooth and uniformly parabolic and the right sides in \eqref{eq:derivativepair} are bounded, so \cite[Theorem~1.2]{Kukuljan} applied in $\Omega_{2r}$ gives, for every $\varepsilon\in(0,1)$,
\begin{equation}
 \frac{\phi_\tau}{\phi_b}
 \in C_p^{1-\varepsilon}\bigl(\overline{\Omega_r}\bigr).
 \label{eq:quotient}
\end{equation}

For small $\ell>0$ the level set $\{\phi=\ell\}\cap Q_r(P_0)$ is a smooth graph $b=\Gamma_\ell(\tau)$, and differentiating $\phi(\tau,\Gamma_\ell(\tau))=\ell$ gives
\[
 \Gamma_\ell'(\tau)=-\frac{\phi_\tau}{\phi_b}\bigl(\tau,\Gamma_\ell(\tau)\bigr).
\]
By \eqref{eq:quotient} the functions $\Gamma_\ell$ are uniformly Lipschitz, and since the parabolic distance between $(\tau,\Gamma_\ell(\tau))$ and $(\tau',\Gamma_\ell(\tau'))$ is at most $C|\tau-\tau'|^{1/2}$, \eqref{eq:quotient} also gives $|\Gamma_\ell'(\tau)-\Gamma_\ell'(\tau')|\leq C|\tau-\tau'|^{(1-\varepsilon)/2}$. Since $\phi(\tau,\Gamma_\ell(\tau))=\ell\to0$, every uniform limit of $\Gamma_\ell$ as $\ell\searrow0$ is $\Gamma$. By the Arzel\`a--Ascoli theorem $\Gamma_{\ell_n}'$ converges uniformly along a subsequence, so $\Gamma\in C^1$ and
\[
 \Gamma'(\tau)=-\left.\frac{\phi_\tau}{\phi_b}\right|_{b=\Gamma(\tau)}.
\]
Hence $\Gamma'$ is locally H\"older continuous of every order below $1/2$, that is, $\Gamma\in C_p^\beta$ for every noninteger $\beta<3$.

\paragraph{Higher order bootstrap}\label{sec:bootstrap}
Suppose that for some noninteger $\beta>2$ the free boundary is of class $C_p^\beta$ and the nearby level curves $b=\Gamma_\ell(\tau)$ have uniformly bounded $C_p^\beta$ norms, which the previous step gives for every $\beta\in(2,3)$. Boundary Schauder estimates for \eqref{eq:gappde} give $\phi\in C_p^\beta$ up to the free boundary. Hence $\phi_\tau\in C_p^{\beta-2}$ and $\phi_b\in C_p^{\beta-1}$, and the right sides in \eqref{eq:derivativepair} belong to $C_p^{\beta-2}$. Then \cite[Theorem~1.3]{Kukuljan} with exponent $\beta-1$ gives $\phi_\tau/\phi_b\in C_p^{\beta-1}$ up to the free boundary, with a local estimate independent of $\ell$. By the formula for $\Gamma_\ell'$, the level curves have uniformly bounded $C_p^{\beta+1}$ norms, and letting $\ell\searrow0$ gives $\Gamma\in C_p^{\beta+1}$. Iterating with noninteger exponents yields $\Gamma\in C^\infty$, with uniform $C_p^k$ bounds for the level curves on smaller neighborhoods for every $k$, and boundary Schauder estimates then make $\phi$ smooth up to the free boundary. Finally, $\Gamma^{-1}=-\sgn(G)h_f$ is Lipschitz by \cref{sec:chart}, so $|\Gamma'|$ is locally bounded away from zero, $\Gamma^{-1}$ is smooth, and so are $h_f$ and the threshold $f$.

\begin{proof}[Proof of \cref{thm:smoothtails}]
Fix a compact interval $I\subset\mathcal T$. The chart of \cref{sec:chart} represents each free boundary over $I$ as a parabolic $C^1$ graph, the full smooth fit argument and \cref{lem:qhopf} give the first order boundary estimates and the nondegeneracy, and the argument of \cref{sec:bh} shows that $a_\lambda,d_\lambda\in C^\infty(I)$. Since $I$ is arbitrary, $a_\lambda,d_\lambda\in C^\infty(\mathcal T)$. The strict inequalities $a_\lambda<A_0$ and $d_\lambda>B_\lambda$ were proved after \eqref{eq:adaptedcurvature}. The same argument without the lower obstacle applies to the ratcheting problem and gives $a_\infty\in C^\infty(\mathcal T)$ and $a_\infty<A_0$. The boundary Schauder estimates of \cref{sec:bootstrap} make $v_\lambda$ smooth up to each free boundary over $\mathcal T$, the full smooth fit gives $(v_\lambda)_z=(v_\lambda)_b=0$ there, and \eqref{eq:quadraticcontact} is the quadratic growth of the gap.
\end{proof}

\subsection{Characteristic points}\label{sec:characteristic}
In the coordinates $(x,b)=(\log z,b)$ the diffusion direction is $(-b/\sigma,-\eta)$.

\begin{proposition}\label{prop:characteristic}
Each $h_f$ in \eqref{eq:adaptedboundary} attains a global maximum, and every maximizer lies in $[-B_*,B_*]$. If $f$ is differentiable at a maximizer, then
\begin{equation}
 \nu_f:=-\frac b\sigma+\eta F_f'(b)
 =\eta h_f'(b)=0,
 \label{eq:normal}
\end{equation}
so the diffusion direction is tangent to the free boundary at such a point.
\end{proposition}

\begin{proof}
By \eqref{eq:logLip}, $h_f(b)\leq F_f(0)+L_*|b|-b^2/(2\sigma\eta)\to-\infty$ as $|b|\to\infty$, so $h_f$ attains a global maximum. By \eqref{eq:hslope}, $h_f$ is strictly increasing on $(-\infty,-B_*)$ and strictly decreasing on $(B_*,\infty)$, so every maximizer lies in $[-B_*,B_*]$, and at a maximizer where $f$ is differentiable $h_f'=0$.
\end{proof}

The free boundary is $\{x=F_f(b)\}$, with normal vector $(1,-F_f'(b))$, so $\nu_f$ is the component of the diffusion direction normal to the free boundary. At a point where the free boundary is smooth, smooth fit and the continuation equation give
\begin{equation}
 \frac12\left(\frac b\sigma-\eta F_f'(b)\right)^2
 \widehat v_{xx}^{\,\mathcal C}
 =
 \begin{cases}
 a_\lambda(b)-A_0, & f=a_\lambda,\\
 d_\lambda(b)-B_\lambda, & f=d_\lambda,
 \end{cases}
 \qquad
 \widehat v(x,b)=v(e^x,b),
 \label{eq:normalcurvature}
\end{equation}
where $\widehat v_{xx}^{\,\mathcal C}$ is the trace of $\widehat v_{xx}$ from the continuation region. If $\nu_f=0$ and the right side of \eqref{eq:normalcurvature} is nonzero, the second derivatives cannot remain bounded up to the free boundary.

Finally, $G''=(2\kappa\sigma-\eta)/(\sigma^2\eta)>0$, so $G$ has at most two zeros, and the parabolic change of variables fails at these points. Hence the part of the free boundary not covered by \cref{thm:smoothtails} lies over $[-B_*,B_*]$ or over a zero of $G$. Since $B_*>0$, the interval $[-B_*,B_*]$ always contains zero, and it may contain $\bar b$ together with most of the stationary distribution of the premium, as in \cref{sec:numerics}. Over this interval we prove the Lipschitz bound of \cref{thm:boundaries} and the smooth fit of \cref{thm:intrinsic}, but no further regularity of the thresholds.

\section{Optimal policies and the ratcheting limit}\label{sec:strategy}
The marginal equation determines the thresholds but not the level of the dual value. We recover that level from the contact reactions, construct the optimal consumption process, and pass to wealth and investment through the budget identity. We then treat ratcheting, the limit of infinite reduction cost, and an explicit example with constant investment opportunities.

\subsection{Contact reactions}\label{sec:switching}
Define the nonnegative contact reactions
\begin{equation}
 m_a(z,b)=(A_0-z)\chi_{\{z\leq a_\lambda(b)\}},\qquad
 m_d(z,b)=(z-B_\lambda)\chi_{\{z\geq d_\lambda(b)\}}.
 \label{eq:reactions}
\end{equation}
They satisfy $0\leq m_a\leq A_0$ and $0\leq m_d\leq z$. Let
\begin{equation}
 (\mathcal R_\delta f)(z,b)=\E_{z,b}\int_0^\infty
 e^{-\delta t}f(Z_t,b_t)dt.\label{eq:resolvent}
\end{equation}
Both reaction potentials are finite by \eqref{eq:linear-moment}.

\begin{lemma}\label{lem:globalgenerator}
The marginal value satisfies, on the whole state space, in distributions and almost everywhere,
\begin{equation}
 (\delta-\A)v_\lambda=1-z-m_a+m_d,\label{eq:globalreaction}
\end{equation}
and it has the resolvent representation
\begin{equation}
 v_\lambda=g-\mathcal R_\delta m_a+\mathcal R_\delta m_d.
 \label{eq:vpotential}
\end{equation}
\end{lemma}
\begin{proof}
In adapted coordinates $U_s,U_b,U_{bb}\in L^\infty_{\rm loc}$. The first weak derivatives vanish almost everywhere on either contact set. For each fixed $s$, $U_b(s,\cdot)$ is Lipschitz and vanishes on the contact set, so $U_{bb}=0$ almost everywhere there. In the continuation region the equation holds classically. Hence $\delta v-\A v$ equals $1-z$ in the continuation region, $\delta\alpha$ on the upper contact set and $-\delta\lambda$ on the lower contact set, almost everywhere, which is \eqref{eq:globalreaction}. In particular the right side has no singular part on the free boundaries.

This regularity also justifies It\^o's formula along the uncontrolled diffusion. Mollify $U$ in $(s,b)$ on a larger compact rectangle. The second order coefficient of $\mathcal K$ is constant and its first order coefficients $\mu$ and $G$ are smooth, so $\mathcal K(U*\rho_\varepsilon)-(\mathcal KU)*\rho_\varepsilon\to0$ locally in $L^1$, with commutators bounded by $C\varepsilon(\|U_s\|_\infty+\|U_b\|_\infty)$. Moreover $\partial_b(U*\rho_\varepsilon)\to U_b$ uniformly on smaller rectangles by \eqref{eq:Ubholder}. Apply It\^o's formula to the mollified functions between a positive time $t_0$ and a stopped finite horizon. The transition density is smooth by \eqref{eq:bracket} and bounded on compact sets for times bounded away from zero, and the law of the process killed at the boundary of the rectangle is dominated by that of the unrestricted process. Thus the drift terms converge in expected time integral and the martingale terms converge locally in $L^2$, and letting $t_0\searrow0$ proves the discounted generator identity for $U$ itself.

We then remove the rectangle and the horizon. The drift in \eqref{eq:globalreaction} is bounded by $C(1+z)$, the discounted terminal value is uniformly integrable on a fixed horizon because it is bounded by a constant plus a multiple of the state-price exponential, and the class D estimate of \cref{sec:pricingappendix} justifies the localization. The remaining terminal expectation tends to zero because $v_\lambda$ is bounded and $\delta>0$. Since $\mathcal R_\delta(1-z)=g$, this gives \eqref{eq:vpotential}. The same argument applies to the ratcheting value, whose growth is linear, using the tail bound $\alpha e^{-\delta T}+ze^{-rT}/r$.
\end{proof}

\subsection{Switching potentials}\label{sec:potentials}
Define
\begin{equation}
 F_0^\lambda=\mathcal R_\delta m_a,\qquad
 F_1^\lambda=F_0^\lambda+v_\lambda
           =g+\mathcal R_\delta m_d.\label{eq:binarypotentials}
\end{equation}
They satisfy
\begin{equation}
 0\leq F_0^\lambda\leq K,
 \qquad \max\{g,-\lambda\}\leq F_1^\lambda\leq\frac1\delta.
 \label{eq:binarybounds}
\end{equation}
The upper bound for $F_1^\lambda$ follows from $F_1^\lambda=\mathcal R_\delta(1-z+m_d)$ and $1-z+m_d\leq1$, and the lower bounds from \eqref{eq:binarypotentials} and $v_\lambda\geq-\lambda$.

\begin{proposition}\label{prop:binarygame}
For $i=0,1$, the potentials \eqref{eq:binarypotentials} are the values
\begin{equation}
 F_i^\lambda(z,b)=\sup_I\E_{z,b}\left[
 \int_0^\infty e^{-\delta t}(1-Z_t)I_tdt
 -\alpha\sum_{\rm entries\ \tau}e^{-\delta\tau}
 -\lambda\sum_{\rm exits\ \tau}e^{-\delta\tau}\right],
 \quad I_{0-}=i,\label{eq:binary}
\end{equation}
of the problem of switching a single consumption layer on and off. An optimal rule switches on when $z\leq a_\lambda(b)$ and off when $z\geq d_\lambda(b)$, and its switches do not accumulate in finite time. The potentials solve the switching system
\begin{align}
 \max\{(\A-\delta)F_0,F_1-F_0-\alpha\}&=0,\label{eq:binaryVI0}\\
 \max\{(\A-\delta)F_1+1-z,F_0-F_1-\lambda\}&=0.\label{eq:binaryVI1}
\end{align}
\end{proposition}
\begin{proof}
By the strong Markov property of the resolvent, the discounted potentials plus their integrated sources are martingales, with source $m_a$ in mode zero and $1-Z+m_d$ in mode one. The obstacle bounds give $F_0\geq F_1-\alpha$ and $F_1\geq F_0-\lambda$. For a finite switching sequence, we stop and telescope these martingales over the holding intervals. Since the reaction flows are nonnegative and the switching inequalities hold in the right direction, the starting potential dominates the running reward minus switching costs plus the discounted terminal potential. Localization and \eqref{eq:binarybounds} justify the expectations. For an admissible infinite sequence we first truncate the number of switches and the horizon and then use the integrability of the costs and \eqref{eq:linear-moment}.

For the threshold rule each switching inequality is an equality, $m_a=0$ before entry in mode zero, and $m_d=0$ before exit in mode one, so every inequality above becomes an equality. On compact sets the two closed contact sets are separated, because $A_0<B_\lambda$. A continuous path has compact range and is uniformly continuous on a finite interval, so it cannot alternate between these sets infinitely often in finite time. The finite horizon equalities and the bounded positive running revenue imply that the discounted switching costs are integrable, and the positive cycle cost $\alpha+\lambda$ covers the case in which one switching cost is zero. The discounted terminal potentials vanish as the horizon tends to infinity, which proves optimality and the equations.
\end{proof}

The potentials $F_i^\lambda$ are continuous, which also follows from the resolvent by splitting off a short time interval and using the smoothness of the transition density afterwards. After \cref{prop:binarygame}, coupling the state flows in \eqref{eq:binary} uniformly over switching rules shows that they are locally Lipschitz, which we use for the integrals over consumption levels below.

The marginal value is also the value of the Dynkin game
\begin{equation}
 v_\lambda(z,b)=\sup_\tau\inf_\rho\E_{z,b}\left[
 \int_0^{\tau\wedge\rho}e^{-\delta t}(1-Z_t)dt
 -\lambda e^{-\delta\tau}\chi_{\{\tau<\rho\}}
 +\alpha e^{-\delta\rho}\chi_{\{\rho\leq\tau,\rho<\infty\}}\right],
 \label{eq:game}
\end{equation}
where an infinite stopping time gives no terminal payment. The two orders of the game give the same value, and the first contact times form a saddle point. Both facts follow from the same stopped generator inequalities or from the construction for separated obstacles in \cite{CK96,Boetius}.

\paragraph{Integral identities for the thresholds}
With $m_a$ and $m_d$ formed from the thresholds, evaluation of \eqref{eq:vpotential} on the two free boundaries gives
\begin{align}
 \alpha&=g(a_\lambda(b))-(\mathcal R_\delta m_a)(a_\lambda(b),b)
                         +(\mathcal R_\delta m_d)(a_\lambda(b),b),\label{eq:boundaryintegral1}\\
 -\lambda&=g(d_\lambda(b))-(\mathcal R_\delta m_a)(d_\lambda(b),b)
                         +(\mathcal R_\delta m_d)(d_\lambda(b),b).\label{eq:boundaryintegral2}
\end{align}
The thresholds are determined by $v_\lambda$ through \eqref{eq:contacts}, and \eqref{eq:boundaryintegral1}--\eqref{eq:boundaryintegral2} are necessary conditions which they satisfy. We do not prove that the two integral equations alone have a unique solution in a class of curves, as is known for some stopping problems in one dimension \cite{PS06}, so they are not used as an independent characterization.

\subsection{The dual value and optimal consumption}\label{sec:layers}
The potentials of the contact reactions fix the level of $J_\lambda$, and their integral over consumption levels gives both the dual value and its optimal control.

\begin{theorem}\label{thm:layers}
Under \cref{ass:main}, the dual value is
\begin{align}
 J_\lambda(y,c,b)={}&\frac{u(c)}\delta-\frac{yc}r
 +\int_c^\infty a^{-\gamma}F_0^\lambda(ya^\gamma,b)da\nonumber\\
 &+\int_0^c a^{-\gamma}
       [F_1^\lambda(ya^\gamma,b)-g(ya^\gamma)]da.
 \label{eq:layerformula}
\end{align}
Both integrals are finite and nonnegative. The value is continuously differentiable in $c$, with
\begin{equation}
 (J_\lambda)_c(y,c,b)=c^{-\gamma}v_\lambda(yc^\gamma,b).
 \label{eq:marginalidentity}
\end{equation}
An optimal consumption process is obtained by nesting the optimal switching rules for all levels $a>0$.
\end{theorem}
\begin{proof}
We compare every consumption path with the policy that keeps $c$ forever. Writing $I_t^a=\chi_{\{c_t>a\}}$,
\begin{align*}
 u(c_t)-u(c)
 &=\int_c^\infty a^{-\gamma}I_t^ada
   -\int_0^c a^{-\gamma}(1-I_t^a)da,\\
 c_t-c&=\int_c^\infty I_t^ada-\int_0^c(1-I_t^a)da.
\end{align*}
By the coarea formula for paths of bounded variation, the positive and negative variations of utility are the integrals of the upward and downward crossings of the levels $a$, weighted by $a^{-\gamma}$. This includes jumps, since a jump crosses every level between its endpoints.

For $a>c$ the layer starts switched off, and its optimized gain over the benchmark is $F_0^\lambda(ya^\gamma,b)$. For $a<c$ it starts switched on, its payoff when it is never switched off is $g(ya^\gamma)$, and the optimized gain is $F_1^\lambda-g$. Optimizing each layer separately gives the right side of \eqref{eq:layerformula} as an upper bound, where the interchange is first made on a finite horizon with levels in $[\varepsilon,R]$ and then extended using admissibility. Finiteness follows from \eqref{eq:binarybounds}:
\begin{align}
 \int_c^\infty a^{-\gamma}F_0^\lambda(ya^\gamma,b)da
 &\leq\frac{Kc^{1-\gamma}}{\gamma-1},\label{eq:upperlayerbound}\\
 0\leq a^{-\gamma}[F_1^\lambda(ya^\gamma,b)-g(ya^\gamma)]
 &\leq\frac yr,\quad 0<a<c.\label{eq:lowerlayerbound}
\end{align}

To attain the bound we drive every layer by the same Brownian path. A layer is switched on when $Y_ta^\gamma\leq a_\lambda(b_t)$ and off when $Y_ta^\gamma\geq d_\lambda(b_t)$. If $a_1<a_2$, then $Y_ta_1^\gamma<Y_ta_2^\gamma$ at all times, so when the higher level switches on the lower one is already on or switches on at the same time, and when the lower level switches off the higher one is already off or switches off at the same time. Hence the rules preserve the ordering $I_t^{a_1}\geq I_t^{a_2}$, and they are of the form $I_t^a=\chi_{\{c_t>a\}}$ for a single consumption process.

On a finite time interval $Y$ stays positive and finite and the factor path has compact range, so the local bounds in \cref{thm:boundaries} keep consumption positive and finite. The construction in \cref{sec:policy} gives a continuous version of locally bounded variation after the initial projection. Restricting the levels to $[\varepsilon,R]$ gives admissible policies which attain the truncated layer integrals, and their dual payoffs are bounded below by that of constant consumption. After a change of sign, negative utility, expenditure and adjustment costs are all nonnegative, so Fatou's lemma shows that the limiting policy has finite expected total cost and a payoff at least the limit of the truncated payoffs. The upper bound gives equality, and the integrability of the weighted variation gives \eqref{eq:utilityTV} by integration by parts.

Finally, differentiating \eqref{eq:layerformula} only at the moving endpoints of integration, which is justified by the continuity of $F_i^\lambda$, gives $(J_\lambda)_c=c^{-\gamma}[g-F_0^\lambda+F_1^\lambda-g](yc^\gamma,b)$, which is \eqref{eq:marginalidentity}.
\end{proof}

\subsection{The consumption regulator and the wealth map}\label{sec:policy}
For a fixed initial shadow price $y$, define the barriers
\begin{equation}
 \underline c_t=\left(\frac{a_\lambda(b_t)}{Y_t}\right)^{\frac1\gamma},\qquad
 \overline c_t=\left(\frac{d_\lambda(b_t)}{Y_t}\right)^{\frac1\gamma}.
 \label{eq:consbarriers}
\end{equation}
The optimal consumption process is the minimal regulator of the constant input $c$ in the interval $[\underline c_t,\overline c_t]$. At time zero it is projected onto $[\underline c_0,\overline c_0]$, and afterwards it increases only when the lower barrier pushes it up and decreases only when the upper barrier pushes it down. In logarithms,
\begin{equation}
 \log c_t=\log c+\Lambda_t^+-\Lambda_t^-,\qquad
 \log \underline c_t\leq\log c_t\leq\log \overline c_t,\label{eq:reflection}
\end{equation}
where the nondecreasing processes $\Lambda^\pm$, which include the initial jump, increase only when $c_t$ is at the corresponding barrier.

The running maximum of $\underline c$ and the running minimum of $\overline c$ cannot be computed separately, because the two interventions interact. A constructive rule alternates between them. Consumption follows the running maximum of the lower barrier until the upper barrier forces a reduction, then the running minimum of the upper barrier until the lower barrier forces an increase, and so on. The barriers are continuous, and by \cref{prop:reduction} their logarithmic distance satisfies
\begin{equation}
 \log \overline c_t-\log \underline c_t\geq\frac1\gamma\log\mathrm{ILA}>0.
 \label{eq:uniformgap}
\end{equation}
On a compact time interval the barriers are bounded and uniformly continuous, and switching from an increasing to a decreasing phase requires a change at least as large as this distance, so there is a positive minimal time between opposite phases. There are therefore finitely many phases, each monotone and bounded, and the regulator is continuous after the initial projection and locally of bounded variation. It agrees layer by layer with the switching rules, and it is the Skorokhod map of \cite{BKR} for a moving interval, applied to a constant input.

\paragraph{The wealth map}
Let
\begin{equation}
 p=\frac{\gamma-1}{\gamma}\in(0,1),\qquad z=yc^\gamma.
\end{equation}
The layer representation gives $J_\lambda(y,c,b)=c^{1-\gamma}j_\lambda(z,b)$. Define
\begin{align}
 \mathcal I_\lambda(z,b)={}&\int_z^\infty F_0^\lambda(\zeta,b)\zeta^{-p-1}d\zeta\nonumber\\
 &+\int_0^z[F_1^\lambda(\zeta,b)-g(\zeta)]\zeta^{-p-1}d\zeta.
 \label{eq:Tmap}
\end{align}
Then
\begin{equation}
 j_\lambda(z,b)=\frac1{\delta(1-\gamma)}-\frac zr
                 +\frac{z^p}{\gamma}\mathcal I_\lambda(z,b),
 \label{eq:jmap}
\end{equation}
and the bounds \eqref{eq:upperlayerbound} and \eqref{eq:lowerlayerbound} give absolute convergence. Moreover $\partial_z \mathcal I_\lambda=(v_\lambda-g)z^{-p-1}$, so $J_\lambda$ is also continuously differentiable in $y$, with wealth map
\begin{align}
 w_\lambda(z,b):=-\partial_zj_\lambda(z,b)
  =\frac1r-\frac p\gamma z^{p-1}\mathcal I_\lambda(z,b)
         -\frac{v_\lambda(z,b)-g(z)}{\gamma z}.
 \label{eq:wealthmap}
\end{align}
In particular
\begin{equation}
 -(J_\lambda)_y(y,c,b)=c\,w_\lambda(yc^\gamma,b),\qquad
 v_\lambda=(1-\gamma)j_\lambda+\gamma z(j_\lambda)_z.
 \label{eq:homogeneousidentities}
\end{equation}
Integrating \eqref{eq:marginalidentity} alone would leave an integration constant depending on $b$, and \eqref{eq:Tmap}--\eqref{eq:wealthmap} determine it.

Where derivatives are classical, $j_\lambda$ satisfies the gradient constraint equation
\begin{equation}
 \max\left\{(\A-\delta)j+\frac1{1-\gamma}-z,\
       (1-\gamma)j+\gamma zj_z-\alpha,\
       -\lambda-(1-\gamma)j-\gamma zj_z\right\}=0.
 \label{eq:jVI}
\end{equation}
We use the layer representation as the definition of $j_\lambda$, and \eqref{eq:jVI} is a consequence of it.

\subsection{Strong duality}\label{sec:primal}
\begin{theorem}\label{thm:strongduality}
For every $x,c>0$ and $b\in\R$,
\begin{equation}
 V_\lambda(x,c,b)=\min_{y>0}\{J_\lambda(y,c,b)+yx\}.
 \label{eq:strongduality}
\end{equation}
The minimum is attained at an interior point $y^*$. The consumption process of \cref{thm:layers} with this multiplier exhausts the budget and is optimal for the primal problem, and its financing wealth is \eqref{eq:financingwealth}. The optimal consumption and the financing portfolio are unique up to null sets.
\end{theorem}
\begin{proof}
For any budget feasible consumption stream, adding $y$ times the budget slack gives $\calU(c_\cdot)\leq J_\lambda(y,c,b)+yx$, which is weak duality.

We next show that the minimum over $y$ is attained at an interior point. Let
\begin{equation}
 \mathcal H_p(b)=\E_b\int_0^\infty e^{-\delta t}
                      (e^{\delta t}H_t)^pdt.
 \label{eq:Kmoment}
\end{equation}
Since $\E M_t^p\leq1$,
\begin{equation}
 0<\mathcal H_p(b)\leq\frac1{\delta-p(\delta-r)}<\infty.
 \label{eq:Kbound}
\end{equation}
The normalization \eqref{eq:normalization} and the pointwise conjugate of $A_0 u$ give
\begin{equation}
 J_\lambda(y,c,b)\leq\alpha u(c)
 -\frac{\gamma A_0^{\frac1\gamma}}{\gamma-1}\mathcal H_p(b)y^p.
 \label{eq:Jupper}
\end{equation}
The constant consumption $y^{-1/\gamma}$, with the corresponding initial adjustment cost, shows that $J_\lambda(y,c,b)\to\alpha u(c)$ as $y\searrow0$. Because $p<1$, \eqref{eq:Jupper} implies $J_\lambda(y,c,b)+yx<\alpha u(c)$ for some small $y>0$, so the endpoint zero is not a minimizer. For $y\geq c^{-\gamma}$ the same policy is an initial reduction and gives
\begin{equation}
 J_\lambda(y,c,b)\geq
 -\left(\frac{1+\delta\lambda}{\delta(\gamma-1)}+\frac1r\right)y^p
 -\lambda u(c).
 \label{eq:Jlowerlarge}
\end{equation}
Thus $J_\lambda+yx\to\infty$ as $y\to\infty$. Since $J_\lambda$ is a supremum of affine functions of $y$, it is convex, and with continuity this gives an interior minimizer.

Let $c^y$ be the dual optimizer of \cref{thm:layers} and $\mathcal B(y)=\E\int_0^\infty H_tc_t^ydt$. For every $y'>0$, $J_\lambda(y',c,b)\geq J_\lambda(y,c,b)-(y'-y)\mathcal B(y)$, so $-\mathcal B(y)$ is a subgradient, and differentiability in $y$ gives $\mathcal B(y)=-(J_\lambda)_y(y,c,b)$. At an interior minimizer $\mathcal B(y^*)=x$. By \cref{prop:budget} the consumption stream can be financed with nonnegative wealth, and its primal utility is $J_\lambda(y^*,c,b)+y^*x$, so weak duality holds with equality.

It remains to prove uniqueness. In utility units the budget functional is convex, because $C(q)$ is strictly convex. If two different optimal utility paths exhausted the budget, their average would have at least the same utility and a strictly smaller cost. Replace its $q_t$ by $\theta q_t$ for $t\geq0$ and some $\theta\in(0,1)$, keeping $q_{0-}$ fixed. This raises consumption, multiplies the downward variation after zero by $\theta$, and makes the initial downward variation at most $\theta$ times its original value, so \eqref{eq:normalization} shows a strict gain in utility, since the part after $\alpha u(c)$ is strictly negative. For $\theta$ close to one this uses only the budget slack, a contradiction. Uniqueness of the portfolio follows from martingale representation.
\end{proof}

The initial multiplier solves the scalar budget equation
\begin{equation}
 w_\lambda(y^*c^\gamma,b)=\frac xc.\label{eq:budgetroot}
\end{equation}
Convexity of $J_\lambda$ makes $w_\lambda$ nonincreasing in its first argument, and a generalized inverse may be used if it has a flat piece. At every state $w_\lambda>0$, since it is the present value of a positive consumption stream divided by current consumption.

Define the wealth thresholds
\begin{equation}
 w_\uparrow(b)=w_\lambda(a_\lambda(b),b),\qquad
 w_\downarrow(b)=w_\lambda(d_\lambda(b),b).
 \label{eq:wealththresholds}
\end{equation}
The ordering is reversed, since a low $z$ means a low shadow price of wealth and therefore a high wealth--consumption ratio. If $x/c>w_\uparrow(b)$, consumption jumps up to $x/w_\uparrow(b)$, and if $x/c<w_\downarrow(b)$, it jumps down to $x/w_\downarrow(b)$. In between, no immediate adjustment takes place. Unlike the constant ratios of \cite{CJK}, these thresholds depend on the risk premium.

\subsection{Wealth and portfolio}\label{sec:portfolio}
Let
\begin{equation}
 \mathcal N_t=\E_t\int_0^\infty H_sc_s^*ds,
 \qquad d\mathcal N_t=\psi_tdW_t.
\end{equation}
The optimal wealth and dollar investment are
\begin{equation}
 X_t^*=\frac{\mathcal N_t-\int_0^t H_sc_s^*ds}{H_t},\qquad
 \pi_t^*=\frac1\sigma\left(\frac{\psi_t}{H_t}+\frac{b_t}{\sigma}X_t^*\right),
 \label{eq:portfolioMR}
\end{equation}
and the cash position is $X_t^*-\pi_t^*$. The conditional version of the envelope argument gives
\begin{equation}
 X_t^*=-(J_\lambda)_y(Y_t,c_t^*,b_t)
 \label{eq:markovwealth}
\end{equation}
at all times. An optimal instantaneous adjustment of consumption leaves the derivative in $y$ unchanged, because the adjustment cost does not depend on $y$ and the two action values coincide. Utility adjustment therefore does not inject cash into wealth.

Where $J_y$ has the required derivatives, matching the diffusion coefficient of \eqref{eq:markovwealth} with that of \eqref{eq:wealth} gives
\begin{equation}
 \pi^*=\frac b{\sigma^2}YJ_{yy}+\frac\eta\sigma J_{yb}.
 \label{eq:portfoliofeedback}
\end{equation}
The sign of the second term comes from the negative factor shock in \eqref{eq:factor}, and this term would be absent if the risk premium were constant. In normalized coordinates,
\begin{equation}
 \frac{\pi^*}{X^*}
 =\frac{b}{\gamma\sigma^2}
  +\frac{b}{\gamma\sigma^2}\frac{v_z}{w}
  +\frac\eta\sigma\frac{j_{zb}}{w},
 \label{eq:hedgedecomposition}
\end{equation}
because $v_z=j_z+\gamma zj_{zz}=-w+\gamma zj_{zz}$. The middle term is the modification of the CRRA position caused by the adjustment costs, and it is nonpositive for $b>0$. The last term hedges changes in the investment opportunity set and has no prescribed sign.

At a regular contact point $v_z=0$, while $j_{zb}$ need not vanish, so the risky share at a consumption threshold need not equal $b/(\gamma\sigma^2)$. The implied risk aversion $bX/(\sigma^2\pi)$ computed from the total portfolio therefore mixes consumption rigidity with intertemporal hedging, and the boundary identity and the U-shaped risky share obtained in \cite{CJK} for constant opportunities need not carry over.

\subsection{Ratcheting}\label{sec:ratchet}
Ratcheting prohibits every reduction of consumption, including one at time zero, while the upward utility cost remains $\alpha$.

\begin{proposition}\label{prop:ratchetPDE}
There is a unique continuous solution of
\begin{equation}
 \max\{\delta v-\A v-1+z,\ v-\alpha\}=0\label{eq:ratchetVI}
\end{equation}
in the class $\alpha-z/r\leq v\leq\alpha$, which we denote by $v_\infty$. It has a positive threshold $a_\infty$ with
\begin{equation}
 \{v_\infty=\alpha\}=\{z\leq a_\infty(b)\},\qquad
 0<a_\infty(b)\leq A_0\min\Bigl\{1,\frac r\delta\Bigr\}.\label{eq:ratchetbound}
\end{equation}
The threshold satisfies \eqref{eq:logLip}, and \cref{thm:intrinsic,thm:smoothtails} apply to it. Its option potential is
\begin{align}
 P&=g-v_\infty
 =\mathcal R_\delta\bigl[(A_0-z)\chi_{\{z\leq a_\infty(b)\}}\bigr],
 \quad0\leq P\leq K,\label{eq:putpotential}\\
 P(z,b)&=\sup_\rho\E_{z,b}\left[e^{-\delta\rho}
 \Bigl(K-\frac{Z_\rho}r\Bigr)\chi_{\{\rho<\infty\}}\right],\label{eq:put}
\end{align}
where the supremum includes $\rho=\infty$, and the first hitting time of $\{z\leq a_\infty(b)\}$ is optimal.
\end{proposition}
\begin{proof}
By \cref{prop:costs}, $v_\lambda$ decreases in $\lambda$, and it is bounded below by \eqref{eq:lowzbound}. The cone bound gives local equicontinuity, so the limit is locally uniform and continuous. On each compact set the lower obstacle is inactive for large $\lambda$, and stability of viscosity solutions gives \eqref{eq:ratchetVI}. It can also be constructed with the upper penalty in \eqref{eq:penaltyPDE} alone.

Uniqueness in the stated class requires comparison for functions of linear growth in $z$, which the barrier $1+x^2+b^4$ in the proof of \cref{prop:PDEexist} does not provide. With $x=\log z$ put
\[
 \Theta(z,b)=1+x^2+b^4+z^{1+\nu}e^{\varrho b^2},
 \qquad \nu,\varrho>0.
\]
A direct computation gives
\[
 \frac{\A(z^{1+\nu}e^{\varrho b^2})}{z^{1+\nu}e^{\varrho b^2}}
 =-\Bigl(2\varrho\kappa-\frac{\nu(1+\nu)}{2\sigma^2}-\frac{2\varrho(1+\nu)\eta}{\sigma}-2\varrho^2\eta^2\Bigr)b^2
 +2\varrho\kappa\bar b\,b+\varrho\eta^2+(\delta-r)(1+\nu).
\]
Since $\kappa\sigma>\eta$ by \cref{ass:main}, we may fix $0<\varrho<(\kappa-\eta/\sigma)/(2\eta^2)$ and then $\nu>0$ so small that the coefficient of $-b^2$ is positive. Together with $\A(x^2+b^4)\leq C(1+x^2+b^4)$ this gives $\A\Theta\leq\mu_0\Theta$ for some $\mu_0>0$, and $(1+z)/\Theta\to0$ as $|x|+|b|\to\infty$. With the barrier $e^{\mu_0(T-t)}\Theta$, the comparison argument for the finite horizon version of \eqref{eq:ratchetVI} extends from compact sets to the whole state space for an upper semicontinuous subsolution $u$ and a lower semicontinuous supersolution $w$ with $u-w\leq C(1+z)$ \cite[Section~5.D]{CIL}. Now let $v$ and $\widetilde v$ be solutions with $\alpha-z/r\leq v,\widetilde v\leq\alpha$. The function $\psi(t,z)=(z/r)e^{-r(T-t)}$ satisfies $-\psi_t-\A\psi+\delta\psi=0$, so by the monotonicity of \eqref{eq:ratchetVI} the function $\widetilde v+\psi$ is a supersolution of the finite horizon problem on $[0,T]$, with terminal value $\widetilde v+z/r\geq v$. Since $v-\widetilde v-\psi\leq z/r$, comparison gives $v\leq\widetilde v+(z/r)e^{-rT}$. Letting $T\to\infty$ and exchanging $v$ and $\widetilde v$ shows $v=\widetilde v$.

The barriers for small $z$ in \cref{thm:boundaries} are uniform as $\lambda\to\infty$ and give a locally positive threshold. Monotonicity and the test at the upper obstacle give $a_\infty\leq A_0$. All derivative and cone bounds pass to the limit. The argument of \cref{lem:globalgenerator}, without the lower contact reaction, gives \eqref{eq:putpotential}. At upper contact $P=K-a_\infty/r\geq0$, which gives the remaining bound in \eqref{eq:ratchetbound}.

Since $v_\infty\leq\alpha$, the function $P$ dominates the exercise payoff $K-z/r$. Its discounted potential is a supermartingale and its reaction vanishes before contact, so optional sampling verifies \eqref{eq:put} and the optimality of the first contact, with integrability from the class D estimate. The bound $P\leq K$ follows from the reaction potential. Finally $P(z,b)\to0$ as $z\to\infty$, because a positive exercise payoff requires $Z\leq rK$, the first such time tends to infinity pathwise as $z$ increases, and $P\leq K$.
\end{proof}

The option in \eqref{eq:put} is a perpetual American put written on the shadow price process $Z$, discounted at rate $\delta$, with generator \eqref{eq:generator}.

\begin{theorem}\label{thm:ratchetpolicy}
The dual ratcheting value and the optimal consumption at multiplier $y$ are
\begin{align}
 J_\infty(y,c,b)&=\frac{u(c)}\delta-\frac{yc}r
     +\int_c^\infty a^{-\gamma}P(ya^\gamma,b)da,\label{eq:Jratchet}\\
 c_t^y&=\max\left\{c,\ \sup_{0\leq s\leq t}
       \left(\frac{a_\infty(b_s)}{Y_s}\right)^{\frac1\gamma}\right\}.
 \label{eq:ratchetpolicy}
\end{align}
The primal problem is feasible exactly when $x\geq c/r$. If $x>c/r$,
\begin{equation}
 V_\infty(x,c,b)=\min_{y>0}\{J_\infty(y,c,b)+yx\},\label{eq:ratchetduality}
\end{equation}
with an interior minimizer, the policy \eqref{eq:ratchetpolicy} exhausts the budget, and the portfolio is given by \eqref{eq:portfolioMR}. At $x=c/r$ the only feasible policy is $c_t=c$, $X_t=c/r$, $\pi_t=0$, with value $u(c)/\delta$. For $x<c/r$ the feasible set is empty and $V_\infty=-\infty$.
\end{theorem}
\begin{proof}
A consumption level that is already included cannot be removed, so its value is $g$. Each higher level is included once, at the optimal exercise time in \eqref{eq:put}, and these times increase with the level. The layer argument of \cref{thm:layers} gives \eqref{eq:Jratchet}, and inverting the ordered entry times gives \eqref{eq:ratchetpolicy}. The bound $P\leq K$ makes the layer integral finite for $\gamma>1$. Since the threshold is continuous, the running maximum is continuous after its possible initial jump and locally bounded, hence locally of bounded variation. The Fatou argument used for finite costs verifies admissibility and attainment.

Every feasible stream satisfies $c_t\geq c$, so its present value is at least $c\E\int_0^\infty H_tdt=c/r$, and depositing $c/r$ in the bank finances $c$ forever. At equality the nonnegative variable $\int_0^\infty H_t(c_t-c)dt$ has zero expectation, which forces the constant stream and its bank financing.

For $x>c/r$ the objective in \eqref{eq:ratchetduality} is at least $u(c)/\delta+(x-c/r)y$, so it diverges as $y\to\infty$. The argument for small $y$ in \cref{thm:strongduality} applies without change, so an interior multiplier exists. Differentiability in $y$ follows from the layer formula, and the envelope and budget argument gives budget exhaustion and equality in \eqref{eq:ratchetduality}. Martingale representation gives the portfolio.
\end{proof}

The upward cost does not affect the primal ratcheting policy. When $dq^-=0$, the normalization becomes
\begin{equation}
 \calU_{\alpha,\infty}=
 \alpha u(c)+(1-\delta\alpha)\E\int_0^\infty e^{-\delta t}u(c_t)dt.
 \label{eq:ratchetalpha}
\end{equation}
In dual coordinates $a_\infty^\alpha=(1-\delta\alpha)a_\infty^0$, which is the case $\lambda=\infty$ of \cref{prop:reduction}, and the optimal multiplier is rescaled by the same factor, so the ratio in \eqref{eq:ratchetpolicy} does not change.

\subsection{The limit of infinite reduction cost}\label{sec:limit}
\begin{theorem}\label{thm:limit}
As $\lambda\to\infty$,
\begin{align}
 v_\lambda&\searrow v_\infty,&
 F_0^\lambda&\searrow P,&F_1^\lambda&\searrow g,&
 J_\lambda&\searrow J_\infty,\label{eq:Flimit}\\
 a_\lambda(b)&\searrow a_\infty(b),&&
 d_\lambda(b)\nearrow\infty.\label{eq:boundlimit}
\end{align}
The convergence of the marginal values and of the increase thresholds is locally uniform, and $d_\lambda(b)\geq1+\delta\lambda$ for all $b$. In adapted coordinates,
\begin{equation}
 (U_\lambda)_b\longrightarrow(U_\infty)_b
 \quad\text{locally uniformly}.\label{eq:directionalconvergence}
\end{equation}
For every $x,c>0$,
\begin{equation}
 V_\lambda(x,c,b)\searrow
 \begin{cases}V_\infty(x,c,b),&x\geq\frac cr,\\
 -\infty,&x<\frac cr.
 \end{cases}\label{eq:Vlimit}
\end{equation}
Fix a multiplier $y>0$ and drive all controls by the same state-price process and factor. Then for every $T<\infty$
\begin{equation}
 \sup_{0\leq t\leq T}|c_t^{\lambda,y}-c_t^{\infty,y}|
 \longrightarrow0\quad\text{almost surely}.\label{eq:controlconvergence}
\end{equation}
\end{theorem}
\begin{proof}
The first limit was obtained in the proof of \cref{prop:ratchetPDE}. Since $v_\lambda\geq v_\infty$ and both are bounded by $\alpha$, every upper contact point of the ratcheting problem is an upper contact point for every finite cost, so $a_\lambda\geq a_\infty$. If the decreasing limit of $a_\lambda(b)$ were larger than $a_\infty(b)$, a point strictly between the two would satisfy $v_\lambda=\alpha$ for all large $\lambda$ but $v_\infty<\alpha$, a contradiction. The common Lipschitz bound then gives locally uniform convergence of the thresholds, and the bound for $d_\lambda$ is \eqref{eq:boundarybounds}.

The upper reaction $m_a$ decreases pointwise to its ratcheting counterpart, and the lower reaction decreases and is bounded by $z\chi_{\{z\geq B_\lambda\}}$. Dominated convergence in the resolvents, with \eqref{eq:linear-moment}, gives $F_0^\lambda\searrow P$ and $F_1^\lambda=g+\mathcal R_\delta m_d\searrow g$, and the uniform layer bounds of \cref{thm:layers} give $J_\lambda\searrow J_\infty$. The estimates \eqref{eq:Ubholder} make $(U_\lambda)_b$ locally equicontinuous, and every subsequential limit is the weak derivative in $b$ of $U_\infty$, which proves \eqref{eq:directionalconvergence}.

Decreasing limits commute with infima, so $J_\lambda\searrow J_\infty$ implies $\lim_\lambda\inf_y(J_\lambda+yx)=\inf_y(J_\infty+yx)$. For $x>c/r$ we apply \eqref{eq:ratchetduality}. For $x=c/r$ we let $y\to\infty$ in \eqref{eq:Jratchet}, where the option integral tends to zero by $P(ya^\gamma,b)\to0$ and the integrable bound $Ka^{-\gamma}$. For $x<c/r$ the term $(x-c/r)y$ tends to $-\infty$ while the option integral stays bounded. This proves \eqref{eq:Vlimit}.

Finally fix a sample path on $[0,T]$. The factor has compact range and $Y$ has a positive minimum and a finite maximum. Local uniform convergence of $a_\lambda$ gives uniform convergence of the lower consumption barriers, which are all bounded above by the barrier for a fixed smaller cost. The upper barriers satisfy
$\inf_{t\leq T}\overline c_t^{\,\lambda}\geq[(1+\delta\lambda)/\sup_{t\leq T}Y_t]^{1/\gamma}\to\infty$,
so for large $\lambda$ no downward intervention occurs before $T$ on this path, including at time zero. The regulator is then the running maximum of its lower barrier above $c$, and since running maxima are nonexpansive in the uniform norm, \eqref{eq:controlconvergence} follows.
\end{proof}

The convergence \eqref{eq:controlconvergence} is stated for a fixed multiplier. For a fixed initial wealth the optimal multiplier depends on $\lambda$, and its convergence gives the corresponding statement for the primal optimal consumption.

\begin{corollary}\label{cor:primalconvergence}
Let $x>c/r$, let $y_\lambda^*$ be a minimizer in \eqref{eq:strongduality} for $\lambda<\infty$, and let $y_\infty^*$ be a minimizer in \eqref{eq:ratchetduality}. Then $y_\infty^*$ is unique, $y_\lambda^*\to y_\infty^*$, and the primal optimal consumption processes $c^{\lambda,*}=c^{\lambda,y_\lambda^*}$ and $c^{\infty,*}=c^{\infty,y_\infty^*}$ satisfy, for every $T<\infty$,
\begin{equation}
 \sup_{0\leq t\leq T}|c_t^{\lambda,*}-c_t^{\infty,*}|\longrightarrow0
 \quad\text{almost surely}.\label{eq:primalconvergence}
\end{equation}
\end{corollary}
\begin{proof}
Write $\varphi_\lambda(y)=J_\lambda(y,c,b)+yx$ for $\lambda\leq\infty$ and $Y_t^y=ye^{\delta t}H_t$.

\emph{Uniqueness.} By the proof of \cref{thm:ratchetpolicy}, $J_\infty$ is convex and differentiable in $y$, and the subgradient argument in the proof of \cref{thm:strongduality} gives $-\partial_yJ_\infty(y,c,b)=\mathcal B_\infty(y)=\E\int_0^\infty H_tc_t^{\infty,y}dt$. If $y_1<y_2$ both minimize $\varphi_\infty$, then $\varphi_\infty$ is constant on $[y_1,y_2]$ and $\mathcal B_\infty(y_1)=\mathcal B_\infty(y_2)=x$. By \eqref{eq:ratchetpolicy}, $c_t^{\infty,y}=\max\{c,(y_2/y)^{1/\gamma}S_t\}$ with $S_t=\sup_{s\leq t}(a_\infty(b_s)/Y_s^{y_2})^{1/\gamma}$, so $c^{\infty,y_1}\geq c^{\infty,y_2}$, with strict inequality where $S_t>c$. Since $x>c/r=\E\int_0^\infty H_tcdt$, the set $\{S_t>c\}$ has positive $dP\otimes dt$ measure, and $\mathcal B_\infty(y_1)>\mathcal B_\infty(y_2)$, a contradiction.

\emph{Compactness.} Fix $\lambda_0<\infty$. For $\lambda\geq\lambda_0$, \eqref{eq:Flimit} gives $J_\infty\leq J_\lambda\leq J_{\lambda_0}$, so $\varphi_\infty(y_\lambda^*)\leq\varphi_\lambda(y_\lambda^*)=V_\lambda(x,c,b)\leq V_{\lambda_0}(x,c,b)$. The function $\varphi_\infty$ is convex and continuous, it is at least $u(c)/\delta+(x-c/r)y$ by the proof of \cref{thm:ratchetpolicy}, and it tends to $\alpha u(c)$ as $y\searrow0$ by the argument in the proof of \cref{thm:strongduality}, while $V_{\lambda_0}(x,c,b)<\alpha u(c)$ by \eqref{eq:Jupper}. Hence the sublevel set $\{\varphi_\infty\leq V_{\lambda_0}(x,c,b)\}$ is a compact interval $[y_{\rm lo},y_{\rm hi}]\subset(0,\infty)$, and it contains $y_\lambda^*$ for all $\lambda\geq\lambda_0$.

\emph{Multipliers.} The convex functions $J_\lambda$ decrease to the continuous function $J_\infty$, so the convergence is uniform on $[y_{\rm lo},y_{\rm hi}]$ by Dini's theorem. Every limit point $\bar y$ of $y_\lambda^*$ therefore satisfies $\varphi_\infty(\bar y)=\lim\varphi_\lambda(y_\lambda^*)=\lim V_\lambda(x,c,b)=V_\infty(x,c,b)$ by \eqref{eq:Vlimit}, and uniqueness gives $\bar y=y_\infty^*$.

\emph{Paths.} Fix a sample path and $T<\infty$. For $y\in[y_{\rm lo},y_{\rm hi}]$ and $\lambda\geq\lambda_0$ the lower consumption barriers are bounded on $[0,T]$ by $\sup_{t\leq T}(a_{\lambda_0}(b_t)/Y_t^{y_{\rm lo}})^{1/\gamma}$, by \cref{prop:costs}, while the upper barriers are at least $[(1+\delta\lambda)/\sup_{t\leq T}Y_t^{y_{\rm hi}}]^{1/\gamma}\to\infty$. As in the proof of \cref{thm:limit}, for large $\lambda$ no reduction occurs on $[0,T]$, and
\[
 c_t^{\lambda,*}=\max\biggl\{c,\ \sup_{0\leq s\leq t}\biggl(\frac{a_\lambda(b_s)}{Y_s^{y_\lambda^*}}\biggr)^{\frac1\gamma}\biggr\},
 \qquad 0\leq t\leq T.
\]
Since $a_\lambda\to a_\infty$ uniformly on the range of the factor on $[0,T]$ and $y_\lambda^*\to y_\infty^*$, the barriers converge uniformly on $[0,T]$ to that of $c^{\infty,*}$, and running maxima are nonexpansive in the uniform norm, which proves \eqref{eq:primalconvergence}.
\end{proof}

\subsection{Constant investment opportunities}\label{sec:constant}
When the risk premium is frozen at $b_t=b_0\ne0$, that is, $\eta=\kappa=0$, everything is explicit. Put $\theta=b_0/\sigma$ and let $m_+>1$ and $m_-<0$ be the roots of
\begin{equation}
 \frac{\theta^2}{2}m(m-1)+(\delta-r)m-\delta=0.
 \label{eq:roots}
\end{equation}
Between its two thresholds $a$ and $d$, the marginal value is
\begin{equation}
 v(z)=D_1\Bigl(\frac za\Bigr)^{m_+}+D_2\Bigl(\frac za\Bigr)^{m_-}+\frac1\delta-\frac zr,
 \label{eq:constantv}
\end{equation}
with value matching and smooth fit
\begin{equation}
 v(a)=\alpha,\quad v'(a)=0,\qquad
 v(d)=-\lambda,\quad v'(d)=0.\label{eq:constantbc}
\end{equation}
The conditions at $a$ give
\begin{align}
 D_1&=\frac{(\alpha-\frac1\delta)m_-+(m_--1)\frac ar}{m_--m_+},\label{eq:D1}\\
 D_2&=\frac{(\alpha-\frac1\delta)m_++(m_+-1)\frac ar}{m_+-m_-},
 \label{eq:D2}
\end{align}
and the wealth--consumption ratio in the inaction interval is
\begin{equation}
 w_\lambda(z)=\frac1r-
 \frac{D_1m_+}{a(1-\gamma+\gamma m_+)}\Bigl(\frac za\Bigr)^{m_+-1}
 -\frac{D_2m_-}{a(1-\gamma+\gamma m_-)}\Bigl(\frac za\Bigr)^{m_--1}.
 \label{eq:constantw}
\end{equation}
These formulas reproduce the analysis of \cite{CJK}, and we use them as a benchmark for the numerics in \cref{sec:numerics}. In the ratcheting problem,
\begin{equation}
 a_\infty=(1-\delta\alpha)\frac{m_+-1}{m_+},\qquad
 D_1^\infty=0,\qquad
 D_2^\infty=\frac{1-\delta\alpha}{\delta(m_--1)},\label{eq:constantinf}
\end{equation}
and
\begin{equation}
 w_\infty(a_\infty)
 =\frac1r\frac{\gamma(m_--1)}{1-\gamma+\gamma m_-},\qquad
 \lim_{z\to\infty}w_\infty(z)=\frac1r.
 \label{eq:constantwealthinf}
\end{equation}
The second limit concerns the ratcheting wealth map for fixed parameters. The reduction threshold for finite cost behaves differently.

\begin{proposition}\label{prop:noncommuting}
In the model with constant investment opportunities,
\begin{equation}
 \lim_{\lambda\to\infty}w_\lambda(d_\lambda)
 =\frac{\gamma(m_+-1)}{r(1-\gamma+\gamma m_+)}
 <\frac1r,\label{eq:noncommuting}
\end{equation}
whereas the solvency floor of the ratcheting problem is $1/r$.
\end{proposition}
\begin{proof}
We have $a_\lambda\to a_\infty>0$, $d_\lambda\to\infty$, $D_1\to0$ and $D_2\to D_2^\infty$. Smooth fit at $d$ gives
\[
 \frac{D_1m_+}{a}\Bigl(\frac da\Bigr)^{m_+-1}
 +\frac{D_2m_-}{a}\Bigl(\frac da\Bigr)^{m_--1}=\frac1r.
\]
The second term tends to zero because $D_2$ and $a^{-1}$ stay bounded and $m_--1<0$, so the first term tends to $1/r$ even though $D_1\to0$. Substitution in \eqref{eq:constantw} gives the limit $1/r-1/[r(1-\gamma+\gamma m_+)]$, which is \eqref{eq:noncommuting}. Its denominator is $1+\gamma(m_+-1)>1$, so the inequality is strict.
\end{proof}

This clarifies the limit of the lower wealth threshold discussed in Appendix~K of \cite{CJK}. The value $1/r$ is both the solvency floor of the ratcheting problem and $\lim_{z\to\infty}w_\infty(z)$, whereas the reduction threshold evaluated at its own moving dual threshold converges to \eqref{eq:noncommuting}. The vanishing coefficient $D_1$ is multiplied by the diverging factor $(d/a)^{m_+-1}$, so the limit $\lambda\to\infty$ cannot be interchanged with evaluation at $z=d_\lambda$, although the limiting policy itself is the ratcheting policy.

For $r=0.03$, $\delta=0.04$, $b_0=0.05$, $\sigma=0.18$, $\alpha=2$ and $\gamma=2$, the limit \eqref{eq:noncommuting} is $15.8614$, whereas $1/r=33.3333$ and the ratcheting increase threshold in wealth units is $w_\infty(a_\infty)=47.0721$. For $\lambda=40$, $160$ and $2560$ the reduction threshold $w_\lambda(d_\lambda)$ equals $16.0613$, $15.8907$ and $15.8617$, and for $\lambda=40$ the increase threshold is $w_\lambda(a_\lambda)=38.1827$.

\section{Numerical illustrations}\label{sec:numerics}
\subsection{Discretization and parameters}
We approximate the uncontrolled diffusion by a Markov chain on a grid in $(x,b)=(\log z,b)$. From each grid point a single standard normal innovation $\xi$ moves both coordinates,
\begin{align}
 x'&=x+\left(\delta-r-\frac{b^2}{2\sigma^2}\right)h
                  -\frac b\sigma\sqrt h\,\xi,\label{eq:step1}\\
 b'&=b+\kappa(\bar b-b)h-\eta\sqrt h\,\xi,\label{eq:step2}
\end{align}
the expectation over $\xi$ is computed by Gauss--Hermite quadrature with five nodes, and the end point is mapped back to the grid by interpolation that is linear in $z$ and in $b$. All weights are nonnegative and the transition matrix $\Pi_h$ has unit row sums, so the scheme is monotone and needs no central differences for the mixed derivative. Interpolation linear in $z$ also reproduces the linear payoff $z$ away from the ends of the grid. For finite $\lambda$ we solve
\begin{equation}
 v_h=\operatorname{clip}_{[-\lambda,\alpha]}
     \left(f_h+e^{-\delta h}\Pi_hv_h\right),\qquad
 f_h(z)=\frac{1-e^{-\delta h}}\delta
                -z\frac{1-e^{-rh}}r,\label{eq:numericalbellman}
\end{equation}
with the exact expected running reward $f_h$. For ratcheting we solve the bounded put problem
\begin{equation}
 P_h(z,b)=\max\left\{\Bigl(K-\frac zr\Bigr)^+,\ e^{-\delta h}(\Pi_hP_h)(z,b)\right\}
 \label{eq:numericalput}
\end{equation}
and set $v_h=g-P_h$, which is more stable than truncating the unbounded negative tail of $v_\infty$. Both maps are contractions with factor $e^{-\delta h}$, and we solve them by policy iteration, in which obstacle rows become identity rows and continuation rows come from the $M$-matrix $I-e^{-\delta h}\Pi_h$. The scheme is monotone, stable and consistent with \eqref{eq:uncontrolledZ}, so it converges to the viscosity solution as $(\Delta x^2+\Delta b^2)/h\to0$ and the truncation is removed \cite{BS91,KD01,FS06}.

We use the illustrative parameters
\begin{equation}
 r=0.03,\quad\delta=0.04,\quad\sigma=0.18,\quad
 \kappa=0.25,\quad\bar b=0.05,\quad\eta=0.03,\quad
 \alpha=2,\quad\gamma=2,\label{eq:calibration}
\end{equation}
which satisfy \cref{ass:main}. The market coefficients are those of the numerical study in \cite{FS26}, whose labor income and preferences differ from ours. The main grid has $449\times101$ points on $[-4.8,6.4]\times[-0.20,0.30]$ with $h=0.05$. At the ends in $x$ we impose the obstacle values, at the ends in $b$ we use clamped interpolation, and each threshold is located at the midpoint in $\log z$ between adjacent contact and continuation nodes.

\subsection{Thresholds and the cost limit}
\Cref{fig:bands} shows the inaction region for $\lambda=40$. It is not symmetric about the long run mean $\bar b$, and the increase threshold peaks near $b=0$. In wealth units the order of the thresholds is reversed, because a high wealth--consumption ratio triggers an increase of consumption and a low ratio a reduction. The wealth thresholds are computed from \eqref{eq:Tmap} by quadrature, with $F_0$ obtained from the discrete entry premium (\cref{sec:discreterec}) and with $F_0=0$ and $F_1-g=0$ beyond the grid. Unlike the dual problem, this step depends on $\gamma$. At $b=\bar b$ the thresholds are $w_\uparrow\approx27.8767$ and $w_\downarrow\approx12.7979$. \Cref{sec:accuracy} explains how these values were obtained and how far they can be trusted.

\begin{figure}[htbp]
\centering
\begin{subfigure}[t]{.49\linewidth}\centering
\includegraphics[width=\linewidth]{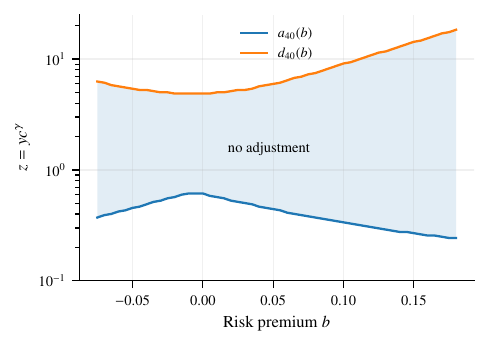}
\caption{Dual thresholds}\label{fig:dualband}
\end{subfigure}\hfill
\begin{subfigure}[t]{.49\linewidth}\centering
\includegraphics[width=\linewidth]{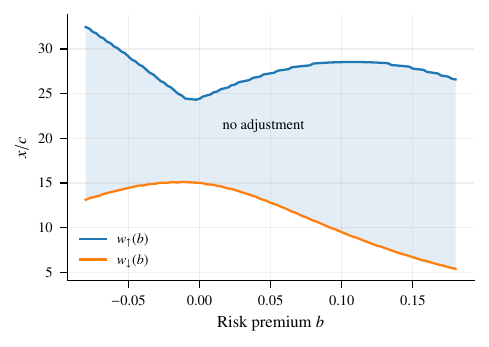}
\caption{Wealth thresholds}\label{fig:wealthband}
\end{subfigure}
\caption{Inaction region for $\lambda=40$. (a) The thresholds $a_{40}(b)$ and $d_{40}(b)$ on a logarithmic scale. (b) The thresholds $w_\uparrow(b)$ and $w_\downarrow(b)$ for the wealth--consumption ratio, computed with the wealth map \eqref{eq:wealthmap} on the domain $[-8.8,10.4]\times[-0.3,0.4]$ with $h=0.025$. The accuracy of these values is discussed in \cref{sec:accuracy}.}
\label{fig:bands}
\end{figure}

As the reduction cost grows, the increase threshold moves down towards its ratcheting limit and the reduction threshold escapes to infinity above the bound $d_\lambda\geq1+\delta\lambda$ (\cref{fig:costlimits}), as \cref{prop:costs,thm:limit} predict. The maximum of the increase threshold near $b=0$ is observed numerically for \eqref{eq:calibration} and is not covered by our results.

\begin{figure}[htbp]
\centering
\begin{subfigure}[t]{.49\linewidth}\centering
\includegraphics[width=\linewidth]{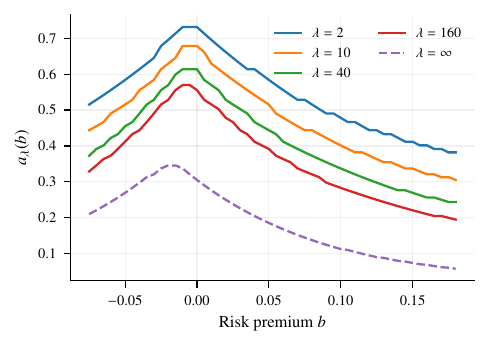}
\caption{Increase thresholds}\label{fig:costlimit}
\end{subfigure}\hfill
\begin{subfigure}[t]{.49\linewidth}\centering
\includegraphics[width=\linewidth]{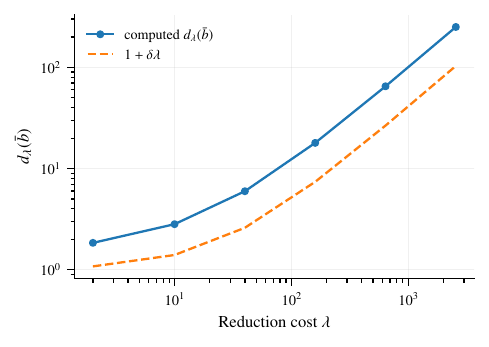}
\caption{Reduction threshold at $\bar b$}\label{fig:escape}
\end{subfigure}
\caption{The limit of infinite reduction cost. (a) Increase thresholds $a_\lambda(b)$ for four costs and for ratcheting, the latter computed from the put problem \eqref{eq:numericalput}. (b) The reduction threshold $d_\lambda(\bar b)$ and the bound $1+\delta\lambda$. The points for $\lambda=640$ and $2560$ use $h=0.10$.}
\label{fig:costlimits}
\end{figure}

\Cref{fig:frozen} compares the ratcheting threshold of the Kim--Omberg model with the threshold obtained when the risk premium is fixed forever at the value on the horizontal axis. The curves cross, so freezing the opportunity set at its current value errs in different directions for different premia. \Cref{fig:noncommuting} illustrates \cref{prop:noncommuting}. As $\lambda$ grows, the reduction threshold $w_\lambda(d_\lambda)$ converges to the limit \eqref{eq:noncommuting}, strictly below the solvency floor $1/r$, although the reduction region disappears in dual coordinates.

\Cref{fig:ratchetpath} shows two simulated consumption paths. Under ratcheting (\cref{fig:ratchetpatha}) consumption is the running maximum of a moving barrier above the inherited level. \Cref{fig:ilapath} uses $\lambda=40$, for which $\mathrm{ILA}=(1+\delta\lambda)/(1-\delta\alpha)=2.8261$. Consumption stays constant as long as it lies between the two barriers, although both barriers move considerably. Early in the fourth year adverse shocks raise the shadow price $Y_t$ and push the upper barrier down to the consumption level. Consumption is cut on five days between $t=3.3611$ and $t=3.4127$, from $1$ to $0.9452$. It then stays at this level until $t=32.6746$, when the lower barrier reaches it and consumption starts to rise. By \eqref{eq:uniformgap} the ratio of the two barriers is at least $\mathrm{ILA}^{1/\gamma}=1.6811$, so small fluctuations of the shadow price leave consumption unchanged.

\begin{figure}[htbp]
\centering
\begin{subfigure}[t]{.49\linewidth}\centering
\includegraphics[width=\linewidth]{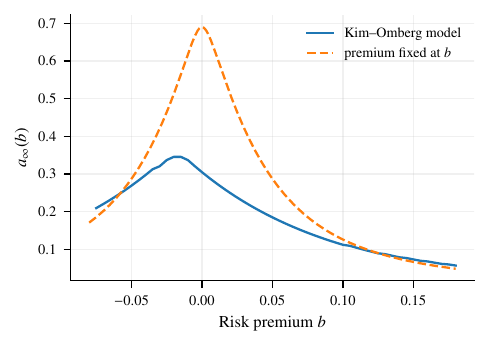}
\caption{Ratcheting threshold}\label{fig:frozen}
\end{subfigure}\hfill
\begin{subfigure}[t]{.49\linewidth}\centering
\includegraphics[width=\linewidth]{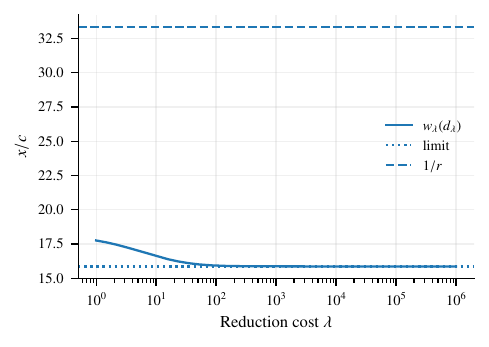}
\caption{Reduction threshold in wealth}\label{fig:noncommuting}
\end{subfigure}
\caption{Comparison with constant investment opportunities. (a) The ratcheting threshold $a_\infty(b)$ in the Kim--Omberg model and in the model with the risk premium permanently fixed at $b$, given by \eqref{eq:constantinf}. (b) For a constant risk premium $b_0=0.05$, the ratio $w_\lambda(d_\lambda)$ at the reduction threshold, its limit \eqref{eq:noncommuting}, and the ratcheting solvency floor $1/r$.}
\label{fig:constantcomparison}
\end{figure}

\begin{figure}[htbp]
\centering
\begin{subfigure}[t]{.49\linewidth}\centering
\includegraphics[width=\linewidth]{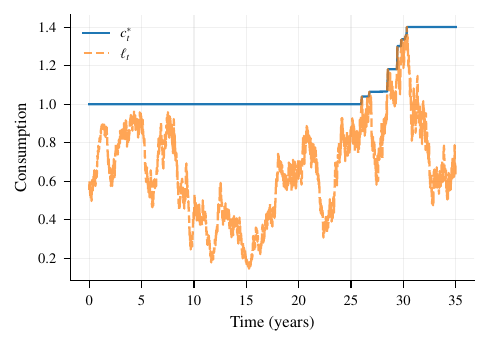}
\caption{Ratcheting}\label{fig:ratchetpatha}
\end{subfigure}\hfill
\begin{subfigure}[t]{.49\linewidth}\centering
\includegraphics[width=\linewidth]{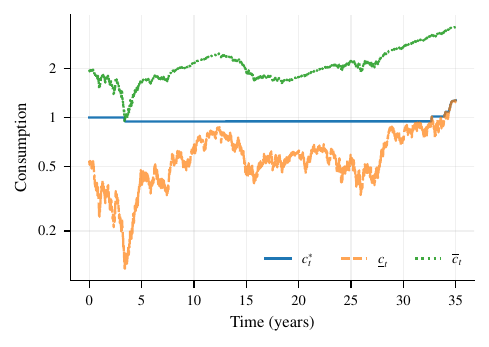}
\caption{$\lambda=40$, $\mathrm{ILA}=2.8261$}\label{fig:ilapath}
\end{subfigure}
\caption{Simulated consumption paths over 35 years with daily steps and $c_{0-}=1$. (a) Ratcheting with $Y_0=0.6$ and seed $20260925$. Consumption $c_t^*$ is the running maximum above $c_{0-}$ of the barrier $\ell_t=[a_\infty(b_t)/Y_t]^{1/\gamma}$. (b) Reduction cost $\lambda=40$ with $Y_0=1.6$, so that $Y_0c_{0-}^\gamma$ lies near the middle of the inaction interval at $\bar b$. Consumption is the regulator \eqref{eq:reflection} of $c_{0-}$ between the barriers $\underline c_t$ and $\overline c_t$ of \eqref{eq:consbarriers}, shown on a logarithmic scale. The seed $20260929$ is the first after $20260925$ for which the path contains both an increase and a reduction.}
\label{fig:ratchetpath}
\end{figure}

\subsection{Regions of smoothness}
For \eqref{eq:calibration} the Gaussian moment equation gives
\begin{equation}
 c_{r,\kappa}=2.2395,\qquad L_*=12.4417,
 \qquad B_*=0.0672,
 \label{eq:numericcone}
\end{equation}
whereas the simpler bound $c_{r,\kappa}\leq(2r)^{-1/2}$ would only give the half width $\eta/\sqrt{2r}=0.1225$. The zeros of $G$ are $-0.0240$ and $0.0990$. The first lies inside $[-B_*,B_*]$, so \cref{thm:smoothtails} applies on
\begin{equation}
 (-\infty,-B_*)\ \cup\ (B_*,0.0990)\ \cup\ (0.0990,\infty).
 \label{eq:numericT}
\end{equation}
The stationary distribution of the premium is normal with mean $\bar b=0.05$ and variance $\eta^2/(2\kappa)=0.0018$. The central interval $[-B_*,B_*]$ therefore contains the long run mean and has stationary probability $0.6544$. \Cref{thm:smoothtails} covers the tails of the distribution but not its center, where we only know that the thresholds are Lipschitz in logarithmic scale. \Cref{fig:adaptedgeometry} plots the computed thresholds in the adapted coordinate. Outside $[-B_*,B_*]$ each curve is strictly monotone, as shown in \cref{sec:chart}, and all three curves attain their maximum inside this interval, in agreement with \cref{prop:characteristic}.

\begin{figure}[htbp]
\centering
\begin{subfigure}[t]{.49\linewidth}\centering
\includegraphics[width=\linewidth]{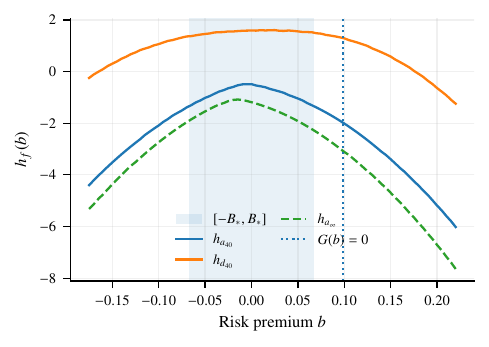}
\caption{Thresholds in adapted coordinates}\label{fig:adaptedgeometry}
\end{subfigure}\hfill
\begin{subfigure}[t]{.49\linewidth}\centering
\includegraphics[width=\linewidth]{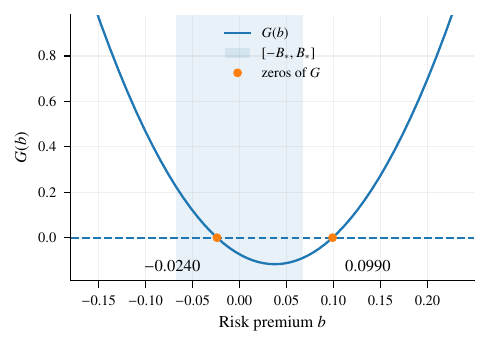}
\caption{Transport coefficient}\label{fig:transport}
\end{subfigure}
\caption{(a) The thresholds in the adapted coordinate $h_f(b)=\log f(b)-b^2/(2\sigma\eta)$ for $f=a_{40},d_{40},a_\infty$. Smoothness is proved outside the shaded interval $[-B_*,B_*]$ and away from the zero $b=0.0990$ of $G$. (b) The coefficient $G(b)$ of \eqref{eq:G}, at whose zeros the parabolic chart of \cref{sec:chart} degenerates.}
\label{fig:smoothregions}
\end{figure}

\subsection{Accuracy}\label{sec:accuracy}
\Cref{tab:grid} reports the thresholds at $b=\bar b$ on three grids of the main scheme. On the main grid the residuals of the discrete Bellman equations are below $10^{-13}$ for $\lambda=40$ and for the ratcheting put, so the discrete equations are solved to machine precision and the error comes from the discretization. We estimate this error in two independent ways.

First, with the risk premium frozen at $b_0=0.05$ the explicit solution of \cref{sec:constant} gives $a=0.3818$ and $d=5.9892$ for $\lambda=40$ and $a_\infty=0.2872$. \Cref{tab:benchmark} shows that the error in the marginal value decreases roughly linearly in $h$. At $h=0.05$ the computed thresholds are $a=0.3916$ and $d=5.8270$, with relative errors $0.0256$ and $-0.0271$.

Second, the coordinates \eqref{eq:straighten} remove the mixed derivative, and an upwind finite difference scheme on a uniform $(s,b)$ grid, with rates $G^\pm/\Delta s$ in $s$ and $\eta^2/(2\Delta b^2)+\mu^\pm/\Delta b$ in $b$, gives a second monotone approximation. We use reflection at the ends in $b$ and obstacle values at the ends in $s$. \Cref{tab:grid} and \cref{fig:independent} compare the two schemes. The increase thresholds of the adapted scheme are consistent with first order convergence in the mesh size, and extrapolation from its two finest grids gives $a_{40}(\bar b)\approx0.4194$ and $a_\infty(\bar b)\approx0.1706$. The main scheme exceeds these values by $0.0581$ and $0.0844$ in relative terms. Its reduction threshold differs from the value on the finest adapted grid by $0.0023$ in relative terms, which is less than one grid cell of $0.0250$ in $\log z$. Enlarging either computational domain leaves the dual thresholds at $\bar b$ unchanged.

The wealth thresholds need two further steps, the potential $F_0$ from the discrete entry premium and the quadrature \eqref{eq:Tmap} with the tail convention $F_0=0$ above the grid and $F_1-g=0$ below it. They inherit the errors of the dual thresholds, and \cref{tab:wealth} checks them in two independent ways. On the dual route we repeat the reconstruction on four grids and on enlarged domains with the same mesh size. The dual thresholds do not depend on the domain, but the wealth thresholds do, because $F_0$ decays slowly as $z\to\infty$ and the truncation affects the integral in \eqref{eq:Tmap}. Only the range of $\log z$ matters. Widening the range of $b$ from $[-0.2,0.3]$ to $[-0.3,0.4]$ changes the wealth thresholds by at most $0.0001$, both on the main domain and on $[-12.8,14.4]$. From the main domain to the largest one, $w_\uparrow(\bar b)$ falls by $0.0406$ and $w_\downarrow(\bar b)$ by $0.0237$ in relative terms, and the corresponding changes beyond $[-8.8,10.4]$ are only $0.0040$ and $0.0023$. Halving $h$ from $0.05$ to $0.025$ raises $w_\uparrow(\bar b)$ by $0.0119$ on the main domain and by $0.0122$ on $[-8.8,10.4]$, because $a_{40}(\bar b)$ decreases towards the extrapolated value of the adapted scheme. The same refinement changes $w_\downarrow(\bar b)$ by at most $0.0013$.

The primal route uses neither the dual problem nor the reconstruction. It solves the problem directly in the variables $(\log(x/c),b)$ by a monotone semi-Lagrangian scheme in discrete time with step $\Delta t$. In each period the investor either keeps consumption and chooses the risky share from a grid, or adjusts consumption at the cost $\hat\lambda$ of \eqref{eq:ila}, and Howard policy iteration solves the discrete problem exactly. With the risk premium frozen at $b_0=0.05$, the finest primal grid gives $38.1909$ and $16.1373$, against the exact thresholds $38.1827$ and $16.0613$ of \cref{sec:constant}. For the Kim--Omberg model the two finest primal grids differ by at most $0.0014$ in relative terms. The finest dual values, with $h=0.025$ on $[-8.8,10.4]$, lie below the finest primal values by $0.0010$ for $w_\downarrow(\bar b)$ and by $0.0214$ for $w_\uparrow(\bar b)$, again in relative terms. The larger gap is explained by the dual threshold $a_{40}(\bar b)$, which at $h=0.025$ is still $0.0384$ above the extrapolated value of the adapted scheme, since a larger $a_{40}$ gives a smaller $w_\uparrow$. We therefore report the finest primal values $w_\uparrow(\bar b)\approx27.8767$ and $w_\downarrow(\bar b)\approx12.7979$. \Cref{fig:wealthband} uses the finest dual reconstruction, so at $\bar b$ its upper curve lies $0.0214$ too low in relative terms.

The figures are therefore reliable for the shape of the thresholds and for the proved orderings. On the main grid individual threshold values carry relative errors as large as $0.0844$, so smaller differences, for instance between thresholds at nearby values of the risk premium, should not be read off the figures.

\begin{table}[htbp]
\centering
\caption{Thresholds at $b=\bar b$. The Markov chain scheme uses $N_1\times N_b$ nodes in $(\log z,b)$ and time step $h$. The adapted scheme uses finite differences with $N_1\times N_b$ nodes in $(s,b)\in[-18,8]\times[-0.25,0.35]$, and the extrapolation from its two finest grids assumes first order convergence in the mesh size.}
\label{tab:grid}
\begin{tabular}{@{}lrrrr@{}}
\toprule
Scheme & $N_1\times N_b$ & $a_{40}(\bar b)$ & $d_{40}(\bar b)$ & $a_\infty(\bar b)$\\
\midrule
Markov chain, $h=0.20$ & $113\times26$ & 0.4499 & 6.0572 & 0.1930\\
Markov chain, $h=0.10$ & $225\times51$ & 0.4607 & 5.9003 & 0.1873\\
Markov chain, $h=0.05$ & $449\times101$ & 0.4437 & 5.9745 & 0.1850\\
\midrule
Adapted & $521\times121$ & 0.3893 & 5.7920 & 0.1749\\
Adapted & $1041\times241$ & 0.4041 & 6.0133 & 0.1727\\
Adapted & $1561\times361$ & 0.4092 & 5.9883 & 0.1720\\
Adapted, extrapolated & & 0.4194 & & 0.1706\\
\bottomrule
\end{tabular}
\end{table}

\begin{table}[htbp]
\centering
\caption{Maximum absolute error of the marginal value on $z\in[0.1,30]$ in the model with constant risk premium $b_0=0.05$.}
\label{tab:benchmark}
\begin{tabular}{@{}rrrr@{}}
\toprule
$h$ & $N_1$ & $\lambda=40$ & Ratcheting\\
\midrule
0.200 & 113 & 0.5923 & 0.3011\\
0.100 & 225 & 0.2935 & 0.1611\\
0.050 & 449 & 0.1500 & 0.0809\\
0.025 & 897 & 0.0618 & 0.0327\\
\bottomrule
\end{tabular}
\end{table}

\begin{table}[htbp]
\centering
\caption{Wealth thresholds at $b=\bar b$ for $\lambda=40$. The dual route uses the Markov chain scheme with time step $h$ on $N_1\times N_b$ nodes in $(\log z,b)$, with $b\in[-0.2,0.3]$ on the domain $[-4.8,6.4]$ and $b\in[-0.3,0.4]$ on the larger ones, followed by the reconstruction \eqref{eq:Tmap}--\eqref{eq:wealthmap}. The primal route uses the semi-Lagrangian scheme with $N_y$ nodes in $\log(x/c)\in[\log3,\log90]$, $b\in[-0.2,0.3]$ with mesh $\eta\sqrt{\Delta t}$ and shares in $[-3.5,5]$ with mesh $0.025$ ($0.0125$ in the last row). The last row is the model with the premium frozen at $0.05$, whose exact thresholds are $38.1827$ and $16.0613$.}
\label{tab:wealth}
\small
\begin{tabular}{@{}llrrr@{}}
\toprule
Route & Grid & Domain in $\log z$ & $w_\uparrow(\bar b)$ & $w_\downarrow(\bar b)$\\
\midrule
Dual & $h=0.2$, $113\times26$ & $[-4.8,6.4]$ & 27.3107 & 12.8497\\
Dual & $h=0.1$, $225\times51$ & $[-4.8,6.4]$ & 27.2285 & 13.0721\\
Dual & $h=0.05$, $449\times101$ & $[-4.8,6.4]$ & 27.9789 & 13.0795\\
Dual & $h=0.025$, $897\times201$ & $[-4.8,6.4]$ & 28.3133 & 13.0630\\
Dual & $h=0.05$, $609\times141$ & $[-6.8,8.4]$ & 27.2018 & 12.8677\\
Dual & $h=0.05$, $769\times141$ & $[-8.8,10.4]$ & 26.9515 & 12.7995\\
Dual & $h=0.05$, $1089\times141$ & $[-12.8,14.4]$ & 26.8431 & 12.7699\\
Dual & $h=0.025$, $1537\times281$ & $[-8.8,10.4]$ & 27.2793 & 12.7847\\
\midrule
Primal & $N_y=400$, $\Delta t=1/48$ & & 27.7561 & 12.7782\\
Primal & $N_y=800$, $\Delta t=1/96$ & & 27.9156 & 12.8095\\
Primal & $N_y=1600$, $\Delta t=1/192$ & & 27.8767 & 12.7979\\
\midrule
Primal, $b$ frozen & $N_y=1600$, $\Delta t=1/192$ & & 38.1909 & 16.1373\\
\bottomrule
\end{tabular}
\end{table}

\begin{figure}[htbp]
\centering
\includegraphics[width=.52\linewidth]{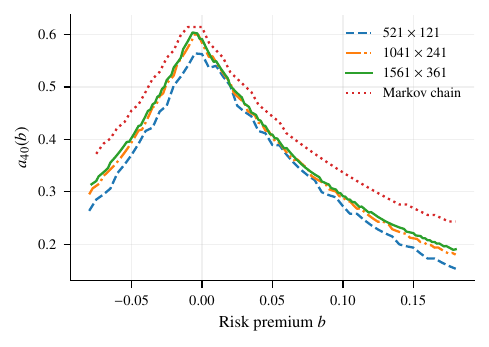}
\caption{The increase threshold $a_{40}(b)$ computed by finite differences in adapted coordinates on grids with $521\times121$, $1041\times241$ and $1561\times361$ nodes, and by the Markov chain scheme with $h=0.05$.}
\label{fig:independent}
\end{figure}

\section{Conclusion}\label{sec:conclusion}
We studied consumption adjustment with utility costs in the complete market of Kim and Omberg through a double obstacle problem for the marginal value of consumption, whose diffusion matrix has rank one. A sharp occupation estimate bounds the sensitivity of the marginal value to the risk premium by its sensitivity to the shadow price. It makes the intervention thresholds Lipschitz in logarithmic scale and gives locally bounded second derivatives and smooth fit along the diffusion direction at every contact point. Away from a central interval of risk premia and at most two further points, the problem reduces to a parabolic obstacle problem in one space variable, and boundary Harnack estimates make the thresholds infinitely differentiable.

The optimal policy needs only the Lipschitz bound and the smooth fit along the diffusion. Potentials of the contact reactions determine the dual value, the regulator between the two consumption barriers gives optimal consumption, and the budget identity gives wealth and the portfolio, which keeps an intertemporal hedging term at the thresholds. Ratcheting is the limit of infinite reduction cost, with solvency floor $c/r$, and optimal consumption from a fixed initial wealth above this floor converges to the ratcheting policy. With constant opportunities the reduction threshold in wealth units converges to a level strictly below this floor.

Two questions remain open. The first is the regularity of the thresholds over the central interval $[-B_*,B_*]$ and at the zeros of $G$, where we neither prove nor exclude singular points. This interval can contain the long run mean and most of the stationary distribution of the premium, and closing the gap requires a regularity theory at characteristic points of operators whose diffusion matrix has rank one. The second is the incomplete market. An independent shock to the risk premium makes the operator elliptic, but the dual problem then involves an optimization over pricing measures, and the budget representation used here is no longer available.

\appendix

\section{Pricing and integrability}\label{sec:pricingappendix}
\subsection{The state-price exponential}
Fix $T<\infty$ and let $\tau_n=\inf\{t:|b_t|\geq n\}$ with $n>|b_0|$. The stopped exponential $M_{t\wedge\tau_n}$ is a true martingale. Under the corresponding measure $Q_n$ on $\mathcal F_T$, before $\tau_n$ the factor satisfies
\[
 db_t=\left[\kappa\bar b-\Bigl(\kappa-\frac\eta\sigma\Bigr)b_t\right]dt
                         -\eta dW_t^{Q_n}.
\]
The usual stopped $L^2$ estimate for this linear equation, with the Burkholder--Davis--Gundy and Gronwall inequalities, gives
\[
 \sup_n\E^{Q_n}\sup_{t\leq T}|b_{t\wedge\tau_n}|^2
 \leq C_T(1+b_0^2),
\]
hence $Q_n(\tau_n\leq T)\leq C_T(1+b_0^2)/n^2\to0$. Since $M_T=M_{T\wedge\tau_n}$ on $\{\tau_n>T\}$,
\[
 \E[M_T\chi_{\{\tau_n>T\}}]=Q_n(\tau_n>T)\longrightarrow1.
\]
As a positive local martingale, $M$ has expectation at most one, and monotone convergence on the events $\{\tau_n>T\}$ gives $\E M_T=1$. The same argument applies from every initial state and on every finite horizon. The market is complete because the single Brownian shock is traded with volatility $\sigma\neq0$.

The family $\{H_\tau:\tau\text{ a stopping time}\}$, with $H_\infty=0$, is uniformly integrable. On a finite horizon this follows from the martingale property of $M$, and for the tail optional sampling gives
\[
 \E[H_\tau\chi_{\{\tau>T\}}]
 \leq e^{-rT}\E[M_\tau\chi_{\{\tau>T\}}]\leq e^{-rT}.
\]
This is the class D control of discounted linear price payoffs used in the stopping arguments.

\subsection{Terminal utility}
Let $q\leq0$ be of locally bounded variation with
\[
 \E\int_0^\infty e^{-\delta t}|q_t|dt<\infty,
 \qquad \E\int_{[0,\infty)}e^{-\delta t}dq_t^-<\infty.
\]
Pathwise integration by parts on $[0,T]$ gives
\[
 e^{-\delta T}q_T-q_{0-}
 =\int_{[0,T]}e^{-\delta t}(dq_t^+-dq_t^-)
          -\delta\int_0^T e^{-\delta t}q_tdt.
\]
Because $q_T\leq0$ and $q_t\leq0$,
\[
 \int_{[0,T]}e^{-\delta t}dq_t^+
 \leq |q_{0-}|+\int_{[0,T]}e^{-\delta t}dq_t^-,
\]
so both discounted variations are integrable on the infinite horizon, and the identity shows that $e^{-\delta T}q_T$ converges in $L^1$. The limit is zero, since a nonzero negative limit would contradict the integrability of $\int e^{-\delta t}|q_t|dt$. This proves \eqref{eq:utilityTV}. When $\lambda=0$ and $\alpha>0$, the discounted upward variation is integrable. The integrability of the negative utility flow gives a sequence $T_n\nearrow\infty$ with $\E[e^{-\delta T_n}|q_{T_n}|]\to0$, and taking expectations in the same identity along this sequence and using monotone convergence shows that the discounted downward variation is integrable. The previous argument then gives the full terminal limit.

\section{The discrete switching reconstruction}\label{sec:discreterec}
In the implementation the marginal game and the switching values are combined as follows. Write $\mathcal Q_h=e^{-\delta h}\Pi_h$, let $E_0=\mathcal Q_hF_0$ and $E_1=f_h+\mathcal Q_hF_1$ be the continuation values without a switch, and set $D=E_1-E_0$. A positive cycle cost excludes two opposite instantaneous switches, and the two optimized values are
\[
 F_0=E_0+(D-\alpha)^+,
 \qquad F_1=E_0+\max\{D,-\lambda\}.
\]
Consequently
\[
 F_1-F_0=\operatorname{clip}_{[-\lambda,\alpha]}D,
 \qquad (I-\mathcal Q_h)F_0=(f_h+\mathcal Q_hv-\alpha)^+.
\]
The first equation is the discrete double obstacle equation, and the second is the linear resolvent equation from which we reconstruct the off value after solving it. Without truncation, the exact transition and reward preserve the value $g=f_h+\mathcal Q_hg$ of a layer that is always on. At the end of the price grid this identity requires an affine tail, since clipping the price would change it, and the code includes this correction and reports the resulting residual.

\section{Remarks on the regularity methods}\label{sec:methods}
\subsection{Elliptic lifting and weighted estimates}\label{sec:audit}
The following observation explains why results for the Heston model cannot be applied directly to the present operator.
\begin{proposition}\label{prop:nolift}
No invertible smooth change of coordinates makes the principal matrix of \eqref{eq:generator} positive definite. Moreover, no exact multiplicative lift $\widetilde u(x,\zeta)=\vartheta(\zeta)u(x)$ into a uniformly elliptic operator preserves the original operator for every test function $u$ while keeping $x$ as the original coordinates.
\end{proposition}
\begin{proof}
An invertible Jacobian $J$ sends a principal matrix $A$ to $JAJ^\top$, which has the same rank. For the lift, comparing the second derivatives in $x$ for arbitrary $u$ shows that the original principal matrix is the $x$ block of the lifted matrix. A principal block of a positive definite matrix is positive definite, whereas here it has a nonzero kernel, and testing the lifted quadratic form on $(\xi,0)$ with $A\xi=0$ gives a contradiction.
\end{proof}

Daskalopoulos and Feehan \cite[Theorem~1.1]{DF} treat a matrix equal to the distance from the boundary of a half-plane times a positive definite matrix, and their weighted norm controls $yD^2u$, $Du$ and $u$. Our matrix has a kernel at every interior point, so we prove the estimate \eqref{eq:Ubbpenalty} directly instead of adapting their weight.

The work \cite{JKO} treats positive variance with correlation strictly between $-1$ and $1$, and its lifting lemma assumes that the spatial matrix is uniformly positive definite on the localized domain. The contact orthant in that proof lies in a fixed slice of parabolic time and rules out a nonzero singular spatial quadratic after the lifting. Here the free boundary $s=h_f(b)$ is first a graph over the transport coordinate, and a Lipschitz cone in $(s,b)$ need not be an open spatial cone at fixed time. We therefore obtain a spatial graph over the set \eqref{eq:smoothset} and verify a positive denominator for the linear boundary Harnack theorem. This route does not use the time monotonicity assumed in the classification of obstacle problems in \cite{AK}.

\subsection{An example of a singular time graph}\label{sec:cusp}
The following elementary example shows why the estimates proved here and general obstacle theory do not combine into smoothness at every point of an arbitrary time graph. Let $\Theta(y)=e^y-1-y$ for $y\geq0$ and define on $\R^2$
\begin{equation}
 u(t,x)=\begin{cases}
 \Theta((|x|+t)^+),&t\leq0,\\
 \E[\Theta(|x+\sqrt2W_t|)]-t,&t>0.
 \end{cases}\label{eq:cuspexample}
\end{equation}
Then $u\geq0$, $u_t\geq0$, $u\in W^{2,1}_\infty$ locally, and
\begin{equation}
 u_t-u_{xx}=-\chi_{\{u>0\}}\quad\text{almost everywhere},\label{eq:cusppde}
\end{equation}
with $|u_x|\leq u_t$ where $u>0$. Nevertheless, the free boundary is the time graph $t=-|x|$, which is not smooth at the origin.

For $t<0$ we differentiate in the two components where $u>0$. There $u_t=e^{|x|+t}-1$, $u_{xx}=e^{|x|+t}$ and $|u_x|=u_t$, and the value and the spatial derivative vanish at the endpoints. At $t=0$ the function $\Theta(|x|)$ has continuous first and second spatial derivatives, with second derivative $e^{|x|}$, so the heat semigroup formula for $t>0$ gives $u_t=\E e^{|x+\sqrt2W_t|}-1>0$ and \eqref{eq:cusppde}. Positivity follows from the nonnegative initial function and this time derivative, and Jensen's inequality applied to the derivative of $\Theta(|x|)$ gives $|u_x|\leq\E(e^{|x+\sqrt2W_t|}-1)=u_t$. The derivative bounds are locally uniform on both sides of $t=0$. The contact set is $\{|x|\leq-t\}$ for $t\leq0$ and empty for $t>0$, which proves the claim about the free boundary.

The example has smooth constant forcing, continuous spatial smooth fit and a monotonicity cone, yet at its tip the spatial contact interval collapses. A spatial thickness or transversality condition therefore cannot be dropped from a higher order parabolic free boundary argument. For the financial problem, \cref{thm:smoothtails} verifies the required geometry on an explicit set, and whether singular points of this kind occur over the central interval is open.

\FloatBarrier

\section*{Data availability}
The code and the numerical data that support the findings of this study are available from the authors upon reasonable request.

\end{document}